\documentclass[12pt,leqno,oneside]{amsart}
\usepackage{mathrsfs,dsfont}
\usepackage{amsmath,amstext,amsthm,amssymb,bbm}
\usepackage{charter}
\usepackage{typearea}
\usepackage{pdfsync}
\usepackage{hyperref}
\usepackage{color}

\usepackage[width=6.4in,height=8.5in]
{geometry}

\newtheorem{Theorem}{Theorem}[section]

\newtheorem{Lemma}[Theorem]{Lemma}
\newtheorem{Corollary}[Theorem]{Corollary}
\theoremstyle{definition}
\newtheorem{Definition}[Theorem]{Definition}
\newtheorem{Example}[Theorem]{Example}
\newtheorem{Remark}[Theorem]{Remark}

\numberwithin{equation}{section}

\newcommand{\norm}[1]{\lVert #1\rVert}
\newcommand{\db}{\overline\partial}

\newcommand{\wi}{\widetilde}
\DeclareMathOperator{\ric}{Ric}
\DeclareMathOperator{\codim}{codim}

\DeclareMathOperator{\supp}{supp}

\newcommand{\cali}[1]{\mathscr{#1}}
\newcommand{\cO}{\cali{O}} 
\newcommand{\cM}{\cali{M}}\newcommand{\cT}{\cali{T}}
\newcommand{\cC}{\cali{C}}
\newcommand{\field}[1]{\mathbb{#1}}

\newcommand{\R}{\field{R}}
\newcommand{\C}{\field{C}}
\newcommand{\N}{\field{N}}

\newcommand{\E}{\mathbb{E}}
\newcommand{\mO}{\mathcal{O}}
\newcommand{\Cdp}{\C^{d_p}}
\newcommand{\hp}{H^0_{(2)}(X,L_p)}

\newcommand{\FS}{{{_\mathrm{FS}}}}

\newcommand{\comment}[1]{}

\begin{document}

\title{Universality results for zeros of random holomorphic sections}

\author{Turgay Bayraktar} 
\thanks{T.\ Bayraktar is partially supported by T\"{U}B\.{I}TAK 
grants B\.{I}DEB 2232/118C006, ARDEB 1001/118F049 
and Science Academy, Turkey BAGEP grant.}
\address{Faculty of Engineering and Natural Sciences, 
Sabanc{\i} University, \.{I}stanbul, Turkey}
\email{tbayraktar@sabanciuniv.edu}

\author{Dan Coman}
\thanks{D.\ Coman is partially supported by the NSF Grant DMS-1700011}
\address{Department of Mathematics, 
Syracuse University, Syracuse, NY 13244-1150, USA}
\email{dcoman@syr.edu}

\author{George Marinescu}
\address{Univerisit\"at zu K\"oln, Mathematisches institut,
Weyertal 86-90, 50931 K\"oln, Germany 
\newline\mbox{\quad}\,Institute of Mathematics `Simion Stoilow', 
Romanian Academy, Bucharest, Romania}
\email{gmarines@math.uni-koeln.de}
\thanks{G.\ Marinescu is partially supported by DFG funded project CRC/TRR 191
and gratefully acknowledges the support of Syracuse University,
where part of this paper was written.}
\thanks{The authors were partially funded through the Institutional Strategy 
of the University of Cologne within the German Excellence Initiative (KPA QM2)}


\date{April 12, 2020}

\begin{abstract}
In this work we prove an universality result regarding the 
equidistribution of zeros of random holomorphic sections 
associated to a sequence of singular Hermitian holomorphic 
line bundles on a compact K\"ahler complex space $X$. 
Namely, under mild moment assumptions, we show that 
the asymptotic distribution of zeros of random holomorphic sections 
is independent of the choice of the probability measure 
on the space of holomorphic sections. 
In the case when $X$ is a compact K\"ahler manifold, 
we also prove an off-diagonal exponential decay estimate for 
the Bergman kernels of a sequence of positive line bundles on $X$.

\end{abstract}

\maketitle

\tableofcontents

\section{Introduction}
\par In this paper we study the asymptotic distribution of zeros of random sequences of 
holomorphic sections of singular Hermitian holomorphic line bundles. We generalize our previous 
results  from \cite{CM11,CM13,CM13b, CMM, B6,B7,B9}
in several directions. 
We consider sequences $(L_p,h_p)$, $p\geq1$, of singular Hermitian
holomorphic line bundles over K\"ahler spaces
instead of the sequence of powers 
$(L^p,h^p)=(L^{\otimes p},h^{\otimes p})$ 
of a fixed line bundle $(L,h)$. Moreover, we endow the vector space 
of holomorphic sections with wide 
classes of probability measures (see condition (B) below and Section \ref{SS:exB}).

\par Recall that by the results of \cite{Ti90} 
(see also \cite[Section\,5.3]{MM07}), if $(X,\omega)$ is 
a compact K\"ahler manifold and $(L,h)$ is a line bundle
such that the Chern curvature form $c_1(L,h)$ equals $\omega$, 
then the normalized Fubini-Study currents $\frac1p\gamma_p$ associated to 
$H^0(X,L^p)$ (see \eqref{e:BFS1}) are smooth for $p$ sufficiently large
and converge in the $\cC^2$ topology to $\omega$. 
This result can be applied to describe the asymptotic distribution 
of the zeros of sequences of Gaussian holomorphic sections. 
Indeed, it is shown in \cite{ShZ99} 
(see also \cite{NoVo:98,DS06,ShZ08,Sh08,DMS}) 
that for almost all sequences 
$\{s_p\in H^0(X,L^p)\}_{p\geq1}$ the normalized 
zero-currents $\frac1p[s_p=0]$ converge weakly to 
$\omega$ on $X$.   
Thus $\omega$ can be approximated 
by various algebraic or analytic objects in the semiclassical limit 
$p\to\infty$\,. Some important technical tools in higher 
dimensions were introduced in \cite{FS95}. 
Using these tools we generalized in 
\cite{CM11,CM13,CM13b,CMM,CMN15,CMN17,DMM} the above results 
to the case of singular positively curved Hermitian metrics $h$.
We note that statistics of zeros of sections and hypersurfaces
have been studied also in the context of real manifolds and
real vector bundles, see e.g.\ \cite{GW14,NS14}.

In this paper we work in the following setting:

\smallskip

(A1) $(X,\omega)$ is a compact (reduced) normal K\"ahler space 
of pure dimension $n$, $X_{\rm reg}$ denotes the set of 
regular points of $X$, and $X_{\rm sing}$ denotes the set 
of singular points of $X$.

\smallskip

(A2) $(L_p,h_p)$, $p\geq1$, is a sequence of holomorphic line 
bundles on $X$ with singular Hermitian metrics $h_p$ whose 
curvature currents verify 
\begin{equation}\label{e:pc}
c_1(L_p,h_p)\geq a_p\,\omega \, \text{ on $X$, where 
$a_p>0$ and } \lim_{p\to\infty}a_p=\infty.
\end{equation}
Let $A_p:=\int_Xc_1(L_p,h_p)\wedge\omega^{n-1}$. 
If $X_{\rm sing}\neq\emptyset$ we also assume that 
\begin{equation}\label{e:domin0}
\exists\,T_0\in\cT(X) \text{ such that } 
c_1(L_p,h_p)\leq A_pT_0\,,\;\forall\,p\geq1\,.
\end{equation}

\par Here $\cT(X)$ denotes the space of positive 
closed currents of bidegree $(1,1)$ on $X$ 
with local plurisubharmonic potentials (see Section \ref{SS:psh}). 
We let $H^0_{(2)}(X,L_p)$ be the Bergman space 
of $L^2$-holomorphic sections of $L_p$ relative to the metric 
$h_p$ and the volume form $\omega^n/n!$ on $X$, 
\begin{equation}\label{e:bs}
H^0_{(2)}(X,L_p)=\left\{S\in H^0(X,L_p):\,
\|S\|_p^2:=\int_{X_{\rm reg}}|S|^2_{h_p}\,\frac{\omega^n}{n!}
<\infty\right\}\,,
\end{equation}
endowed with the obvious inner product.  
For $p\geq1$, let $d_p=\dim H^0_{(2)}(X,L_p)$ and 
let $S_1^p,\dots,S_{d_p}^p$ be an orthonormal 
basis of $H^0_{(2)}(X,L_p)$.

\par Now, we describe the randomization on $\hp$. 
Using the above orthonormal bases
we identify the spaces $\hp\simeq \C^{d_p}$ and endow them 
with probability measures  $\sigma_p$ verifying the following 
moment condition: 

\smallskip

(B) There exist a constant $\nu\geq1$ and for every $p\geq1$ constants 
$C_p>0$ such that  
\[\int_{\C^{d_p}}\big|\log|\langle a,u\rangle|\,\big|^\nu\,d\sigma_p(a)
\leq C_p\,,\,\text{for any $u\in\C^{d_p}$ with $\|u\|=1$}\,.\] 

\smallskip

We remark that the probability space $(\hp,\sigma_p)$ depends in general on the choice of 
the orthonormal basis (used for the identification $\hp\simeq \C^{d_p}$). 
However, it follows from Theorem \ref{th1} below that the global distribution 
of zeros of random holomorphic sections does not depend on the choice 
of the orthonormal basis. 

General classes of measures $\sigma_p$ that satisfy condition (B) are given in 
Section \ref{SS:exB}. Important examples are provided by the Gaussians 
(see Section \ref{s:Gauss}) and the Fubini-Study volumes (see Section \ref{s:FSvol}), 
{\em which verify (B) for every $\nu\geq1$ with a constant $C_p=\Gamma_\nu$ independent of $p$}. 
For such measures Theorem \ref{th1} below takes a particularly nice form. 
We note that for the measures $\sigma_p$ from Sections \ref{s:Gauss}, \ref{s:FSvol} and \ref{s:sphere} (area measure of spheres), 
the probability space $(\hp,\sigma_p)$ does not depend on the choice of the orthonormal basis, since these measures are unitary invariant.
In Section \ref{s:tail} we show that measures with \textit{heavy tail probability} 
(see condition (B1) therein) and \textit{small ball probability} (see condition (B2) therein) 
verify assumption (B). 
We also stress that random holomorphic sections with i.i.d. coefficients whose 
distribution has bounded density and logarithmically decaying tails arise as a 
special case (cf.\ Lemma \ref{iid}). Moreover, locally moderate measures with 
compact support are also among the examples of such measures (cf. Lemma \ref{moderate}).

Given a section $s\in H^0(X,L_p)$ we denote by $[s=0]$
the current of integration over the zero divisor of $s$.
The expectation current $\E[s_p=0]$ of the current-valued random variable 
$H^0_{(2)}(X,L_p)\ni s_p\mapsto[s_p=0]$ is defined by
$$\big\langle \E[s_p=0],\Phi\big\rangle=
\int\limits_{H^0_{(2)}(X,L_p)}\big\langle[s_p=0],\Phi\big\rangle\,d\sigma_p(s_p),$$
where $\Phi$ is a $(n-1,n-1)$ test form on $X$. 
We consider the product probability space
\begin{equation}\label{e:calH}
(\mathcal{H},\sigma)=
\left(\prod_{p=1}^\infty H^0_{(2)}(X,L_p),\prod_{p=1}^\infty\sigma_p\right). 
\end{equation}
The following result gives the distribution
of the  zeros of random
sequences of holomorphic sections of $L_p$, as well as the convergence
in $L^1$ of 
the logarithms of their pointwise norms. Note that by the Lelong-Poincar\'e formula
(see \eqref{e:LP}) the latter are the potentials of the currents
of integration on the zero sets, thus their convergence in 
$L^1$ implies the weak convergence of the zero-currents. 
\begin{Theorem}\label{th1}
Assume that $(X,\omega)$, $(L_p,h_p)$ and $\sigma_p$ 
verify the assumptions (A1), (A2) and (B). 
Then the following hold:

\smallskip
\noindent
(i) If $\displaystyle\lim_{p\to\infty}C_pA_p^{-\nu}=0$ 
then $\displaystyle\frac{1}{A_p}\big(\E[s_p=0]-c_1(L_p,h_p)\big)\to 0$\,, 
as $p\to \infty$, in the weak sense of currents on $X$. 

\smallskip
\noindent
(ii) If $\displaystyle\liminf_{p\to\infty}C_pA_p^{-\nu}=0$ 
then there exists a sequence of natural numbers $p_j\nearrow\infty$ 
such that for $\sigma$-a.\,e.\ sequence $\{s_p\}\in\mathcal{H}$ we have
\[
\frac{1}{A_{p_j}}\log|s_{p_j}|_{h_{p_j}}\to0\,,\,\;
\frac{1}{A_{p_j}}\big([s_{p_j}=0]-c_1(L_{p_j},h_{p_j})\big) \to0\,,\,
\text{ as $j\to\infty$,}
\]
in $L^1(X,\omega^n)$, respectively in the weak sense of currents on $X$.

\smallskip
\noindent
(iii) If $\displaystyle\sum_{p=1}^{\infty}C_pA_p^{-\nu}<\infty$ 
then for $\sigma$-a.\,e.\ sequence
$\{s_p\}\in\mathcal{H}$ we have 
\[
\frac{1}{A_p}\log|s_p|_{h_p}\to0\,,\,\;
\frac{1}{A_p}\big([s_p=0]-c_1(L_p,h_p)\big) \to0\,,\,\text{ as $p\to\infty$,}
\]
in $L^1(X,\omega^n)$, respectively in the weak sense of currents on $X$.
\end{Theorem}
\begin{Remark}\label{R:indp}
If the measures $\sigma_p$ verify condition (B) with constants 
$C_p=\Gamma_\nu$ independent of $p$ then the hypothesis of $(i)$ 
(and hence of $(ii)$), $\lim_{p\to\infty}\Gamma_\nu A_p^{-\nu}=0$, 
is automatically verified since by \eqref{e:pc},
\[A_p\geq a_p\int_X\omega^n\,,\,\text{ so $A_p\to\infty$ as $p\to\infty$.}\]
Moreover, the hypothesis of $(iii)$ takes the simpler form $\sum_{p=1}^{\infty}A_p^{-\nu}<\infty$.
\end{Remark}
An important ingredient in the proof of Theorem \ref{th1} is the asymptotic behavior of the 
Bergman kernel functions $P_p$ of the spaces $H^0_{(2)}(X,L_p)$ (see \eqref{e:BFS1} for the 
definition) established in \cite[Theorem 1.1]{CMM}: namely, one has that 
\[\frac{1}{A_p}\,\log P_p\to0\, \text{ as $p\to\infty$ in $L^1(X,\omega^n)$.}\] 
Theorem \ref{th1} will follow from this using Theorem \ref{th1gen}, which shows, 
under very general assumptions, that the equidistribution of zeros of random holomorphic 
sections is a consequence of the asymptotic behavior of the Bergman kernel (see \eqref{e:Bka}). 
A similar approach was used in a different context in \cite[Theorems 1.1 and 1.2]{CM11}.
\par If $(L_p,h_p)=(L^p,h^p)$, where $(L,h)$ is a fixed singular 
Hermitian holomorphic line bundle, Theorem \ref{th1} gives analogues of the 
equidistribution results from \cite{ShZ99,CM11,CM13,CM13b, CMM} 
for Gaussian ensembles and \cite{DS06,B6,B7,BL13} for non-Gaussian ensembles 
on compact normal K\"ahler spaces. Note that in this case
hypothesis \eqref{e:domin0} is automatically verified as 
$c_1(L^p,h^p)=p\,c_1(L,h)$, so we can take 
$T_0=c_1(L,h)/\|c_1(L,h)\|$, 
where $\|c_1(L,h)\|:=\int_Xc_1(L,h)\wedge\omega^{n-1}$.
We formulate here a corollary in this situation, 
for further variations of Theorem \ref{th1} see Section \ref{S:equidist}.

\begin{Corollary}\label{C:Lp1}
Let $(X,\omega)$ be a compact normal K\"ahler space and $(L,h)$ 
be a singular Hermitian holomorphic line bundle on $X$ such that 
$c_1(L,h)\geq\varepsilon\omega$ for some $\varepsilon>0$. 
For $p\geq1$ let $\sigma_p$ be probability measures on $H^0_{(2)}(X,L^p)$
satisfying condition (B).
Then the following hold:

\smallskip
\noindent
(i) If $\displaystyle\lim_{p\to\infty}C_p\,p^{-\nu}=0$ 
then $\displaystyle\frac{1}{p}\,\E[s_p=0]\to c_1(L,h)$\,, 
as $p\to \infty$, weakly on $X$. 

\smallskip
\noindent
(ii) If $\displaystyle\liminf_{p\to\infty}C_p\,p^{-\nu}=0$ 
then there exists a sequence of natural numbers $p_j\nearrow\infty$ 
such that for $\sigma$-a.\,e.\ sequence $\{s_p\}\in\mathcal{H}$ we have
as $j\to\infty$,
\[\frac{1}{p_j}\,\log|s_{p_j}|_{h^{p_j}}\to0
\,\;\text{in $L^1(X,\omega^n)$} \,,\:\:
\frac{1}{p_j}\,[s_{p_j}=0]\to c_1(L,h)\,,\,
\text{weakly on $X$}.\]

\smallskip
\noindent
(iii) If $\displaystyle\sum_{p=1}^{\infty}C_p\,p^{-\nu}<\infty$ 
then for $\sigma$-a.\,e.\ sequence
$\{s_p\}\in\mathcal{H}$ we have as $p\to\infty$,
\[\frac{1}{p}\,\log|s_p|_{h^p}\to0
\,\;\text{in $L^1(X,\omega^n)$} \,,\:\:
\frac{1}{p}\,[s_p=0] \to c_1(L,h)\,,\,
\text{weakly on $X$}.\]
\end{Corollary}

It is by now well established that the off-diagonal decay of the
Bergman/Szeg\H{o} kernel for powers $L^p$ of a line bundle $L$ implies the 
asymptotics of the variance current and variance number
for zeros of random holomorphic sections of $L^p$, 
cf.\ \cite{B9,STr,ShZ08}. Note also that the Bergman kernel
provides the 2-point correlation function for the 
determinantal random point process defined by the Bergman projection
\cite[\S 6.1]{Berman}.  

We wish to consider here the off-diagonal decay for Bergman kernels
of a sequence $L_p$ satisfying \eqref{e:pc}. We expect that this will have 
applications in obtaining a Central Limit Theorem for smooth linear statistics of zero divisors.
To state our result, let us introduce the relevant definitions.
We consider the situation where $X$ is smooth 
and the Hermitian metrics $h_p$ on $L_p$ are also smooth.
Let $L^2(X,L_p)$ be the space of $L^2$ integrable
sections of $L_p$ with respect to the metric $h_p$
and the volume form $\omega^n/n!$\,. We assume now that $h_p$ 
is smooth, hence $H^0_{(2)}(X,L_p)=H^0(X,L_p)$.
Let $P_p:L^2(X,L_p)\to H^0(X,L_p)$ be the orthogonal projection.
The Bergman kernel $P_p(x,y)$ is defined as the integral 
kernel of this projection, see \cite[Definition 1.4.2]{MM07}.
Let $d_p=\dim H^0(X,L_p)$ and $(S^p_j)_{j=1}^{d_p}$
be an orthonormal basis of $H^0(X,L_p)$.
We have
\[P_p(x,y)=\sum_{j=1}^{d_p} S^p_j(x)\otimes S^p_j(y)^*\in L_{p,x}\otimes L_{p,y}^*\,,\]
where $S^p_j(y)^*=\langle\,\cdot\,,S^p_j(y) \rangle_{h_p}\in L_{p,y}^*$.
We set $P_p(x):=P_p(x,x)$.

The next result provides the exponential off-diagonal
decay of the Bergman kernels $P_p(x,y)$ for 
sequences of positive line bundles $(L_p,h_p)$. 
Adapting methods from \cite{Lin01,Be03} we prove the following:  

\begin{Theorem}\label{T:Bergman}
Let $(X,\omega)$ be a compact K\"ahler manifold of dimension $n$
and $(L_p,h_p)$, $p\geq1$, be a sequence of holomorphic line bundles on $X$ 
with Hermitian metrics $h_p$ of class $\cC^3$ whose curvature forms verify \eqref{e:pc}.
Assume that
\begin{equation}\label{e:3d}
\varepsilon_p:=\|h_p\|_3^{1/3}a_p^{-1/2}\to0\;
\text{ as }p\to\infty\,.
\end{equation}
Then there exist constants $C,T>0$, $p_0\geq1$, such that
for every $x,y\in X$ and $p>p_0$ we have 
\begin{equation}\label{e:ed}
\big|P_p(x,y)\big|^2_{h_p}
\leq C\exp\!\big(\!-T\sqrt{a_p}\,d(x,y)\big)\frac{c_1(L_p,h_p)^n_x}{\omega^n_x}\,
{\frac{c_1(L_p,h_p)^n_y}{\omega^n_y}}\,\cdot
\end{equation}
\end{Theorem}

Here $\|h_p\|_3$ denotes the sup-norm of the derivatives of $h_p$ of order at most 
three with respect to a reference cover of $X$ as defined in Section \ref{SS:refcov}, and 
$d(x,y)$ denotes the distance on $X$ induced by the K\"ahler metric 
$\omega$. We also recall that, in the hypotheses of Theorem \ref{T:Bergman}, 
the first order asymptotics of the Bergman kernel function $P_p(x)=P_p(x,x)$ 
was obtained in \cite[Theorem 1.3]{CMM} 
(see Theorem \ref{T:B1} below).

The situation when $(L_p,h_p)=(L^p,h^p)$ was intensively studied.
Let $(L_p,h_p)=(L^p,h^p)$, such that
there exists a constant $\varepsilon>0$ with
\begin{equation}\label{eq:0.6}
c_1(L,h)\geqslant\varepsilon\omega\,.
\end{equation}
Then $a_p=p\varepsilon$ and $\|h_p\|_3\lesssim p$ so \eqref{e:pc} and
\eqref{e:3d} are satisfied, thus \eqref{e:ed} holds in this case, and is a particular case of
\eqref{eq:0.7} below.
Namely, by \cite[Theorem 1]{MM15},
there exist $T >0$, $p_{0}>0$
so that for any $k\in \N$, there exists $C_k>0$ such that 
for any $p\geqslant p_{0}$, $x,y\in X$, we have 
\begin{equation}\label{eq:0.7}
\left| P_p(x,y)\right|_{\cC^k} \leqslant C_k \, p^{n+\frac{k}{2}}
\, \exp\!\left(- T\,\sqrt{p}\, d(x,y)\right).
\end{equation}
In \cite[Theorem 4.18]{DLM06}, \cite[Theorem 4.2.9]{MM07},
a refined version of \eqref{eq:0.7} was obtained, i.e., the asymptotic expansion
of $P_p(x,y)$ for $p\to +\infty$ 
with an exponential estimate of the remainder.
The estimate \eqref{eq:0.7} holds actually for complete K\"ahler manifolds with
bounded geometry and for the Bergman kernel of the bundle
$L^p\otimes E$, where $E$ is a fixed holomorphic Hermitian vector bundle.

Assume that $X=\C^n$ with the Euclidean metric,
$L=\C^{n+1}$ and $h=e^{-\varphi}$ where 
$\varphi:X\to\R$ is a smooth plurisubharmonic function  
such that \eqref{eq:0.6} holds.
Then the estimate \eqref{eq:0.7}  with $k=0$ was
obtained by  \cite{Christ91} for $n=1$ and 
\cite{Delin}, \cite{Lin01} for $n\geqslant 1$ 
(cf.\ also \cite{Be03}).
In \cite[Theorem 2.4]{B9} 
the exponential decay was obtained for a family of weights having 
super logarithmic growth at infinity.

Assume that $X$ is a compact K\"ahler manifold,
$c_1(L,h)=\omega$ and take $k=0$ and $d(x,y)>\delta>0$.
Then \eqref{eq:0.7} was obtained in \cite[Theorem 2.1]{LuZ13} (see also \cite{Berman})
and a sharper estimate than \eqref{eq:0.7}
is due to Christ \cite{Christ13}.

The paper is organized as follows. After introducing necessary
notions in Section \ref{S:prelim}, we prove Theorem \ref{T:Bergman} 
in Section \ref{S:berg}. In Section \ref{S:equidist}
we prove Theorem \ref{th1} and we provide
examples of measures satisfying condition (B) showing
how Theorem \ref{th1} transforms in these cases.


\section{Preliminaries}\label{S:prelim}
\subsection{Plurisubharmonic functions and currents on analytic 
spaces}\label{SS:psh} 
Let $X$ be a complex space. A chart $(U,\tau,V)$ on $X$ is a 
triple consisting of an open set $U\subset X$, a closed complex 
space $V\subset G\subset\C^N$ in an open set $G$ of $\C^N$ 
and a biholomorphic map $\tau:U\to V$ (in the category of complex 
spaces). The map $\tau:U\to G\subset\C^N$ is called a local 
embedding of the complex space $X$. We write
\[X=X_{\rm reg}\cup X_{\rm sing}\,,\]
where $X_{\rm reg}$ (resp.\ $X_{\rm sing}$) is the set of 
regular (resp.\ singular) points of $X$. 
Recall that a reduced complex space $(X,\cO)$ is called normal
if for every $x\in X$ the local ring $\cO_x$ is integrally closed in 
its quotient field $\cM_x$. Every normal complex space is 
locally irreducible and locally pure dimensional, 
cf.\ \cite[p.\,125]{GR84}, $X_{\rm sing}$ is a closed complex 
subspace of $X$ with $\codim X_{\rm sing}\geq2$. Moreover, 
Riemann's second extension theorem holds on normal complex 
spaces \cite[p.\,143]{GR84}. In particular, every holomorphic 
function on $X_{\rm reg}$ extends uniquely to a holomorphic 
function on $X$.   

Let $X$ be a complex space. A continuous (resp.\ smooth) function 
on $X$ is a function $\varphi:X\to\C$ such that for every 
$x\in X$ there exists a local embedding $\tau:U\to G\subset\C^N$ 
with $x\in U$ and a continuous (resp.\ smooth) function 
$\widetilde\varphi:G\to\C$ such that 
$\varphi|_U=\widetilde\varphi\circ\tau$. 
  
A \emph{(strictly) plurisubharmonic (psh)} function on $X$ 
is a function $\varphi:X\to[-\infty,\infty)$ such that for every 
$x\in X$ there exists a local embedding $\tau:U\to G\subset\C^N$ 
with $x\in U$ and a (strictly) psh function 
$\widetilde\varphi:G\to[-\infty,\infty)$ such that 
$\varphi|_U=\widetilde\varphi\circ\tau$. If $\widetilde\varphi$ 
can be chosen continuous (resp.\ smooth), then $\varphi$ is 
called a continuous (resp.\ smooth) psh function. 
The definition is independent of the chart, as is seen from 
\cite[Lemma\,4]{Nar62}. 
The analogue of Riemann's second extension theorem 
for psh functions holds on normal complex spaces 
\cite[Satz\,4]{GR56}. In particular, every psh function on 
$X_{\rm reg}$ extends uniquely to a psh function on $X$. 
We let $PSH(X)$ denote the set of psh functions on $X$, 
and refer to \cite{GR56}, \cite{Nar62}, \cite{FN80}, \cite{D85} 
for the properties of psh functions on $X$. We recall here that 
psh functions on $X$ are locally integrable with respect to the 
area measure on $X$ given by any local embedding 
$\tau:U\to G\subset\C^N$ \cite[Proposition 1.8]{D85}.

Let $X$ be a complex space of pure dimension $n$. 
We consider currents on $X$ as defined in \cite{D85}
and we denote by $\mathcal D'_{p,q}(X)$ the space of 
currents of bidimension $(p,q)$, or bidegree $(n-p,n-q)$ on $X$.     
In particular, if $v\in PSH(X)$ then 
$dd^cv\in\mathcal D'_{n-1,n-1}(X)$ is positive and closed. 
Let $\cT(X)$ be the space of positive closed currents 
of bidegree $(1,1)$ on $X$ which have local psh potentials: 
$T\in\cT(X)$ if every $x\in X$ has a neighborhood $U$
(depending on $T$) such that there exists a psh function $v$ on 
$U$ with $T=dd^cv$ on $U\cap X_{\rm reg}$. 
Most of the currents considered here, such as the curvature currents 
$c_1(L_p,h_p)$ and the Fubini-Study currents $\gamma_p$, 
belong to $\cT(X)$. 
A K\"ahler form on $X$ is a current $\omega\in\cT(X)$ 
whose local potentials 
extend to smooth strictly psh functions in local embeddings of $X$ 
to Euclidean spaces. We call $X$ a K\"ahler space if $X$ admits
a K\"ahler form (see also \cite[p.\,346]{Gra:62}, 
\cite{O87}, \cite[Sec. 5]{EGZ09}). 

\subsection{Singular Hermitian holomorphic line bundles 
on analytic spaces}\label{SS:lb} 
Let $L$ be a holomorphic line bundle on a normal K\"ahler space
$(X,\omega)$. 
The notion of singular Hermitian metric $h$ on $L$ is defined 
exactly as in the smooth case 
(see \cite{D90}, \cite[p.\,97]{MM07})): if $e_\alpha$
  is a holomorphic frame of $L$ over an open set 
  $U_\alpha\subset X$
   then $|e_\alpha|^2_h=e^{-2\varphi_\alpha}$ where 
   $\varphi_\alpha\in L^1_{loc}(U_\alpha,\omega^n)$.
 If $g_{\alpha\beta}=e_\beta/e_\alpha\in\mathcal O^*_X
 (U_\alpha\cap U_\beta)$ 
    are the transition functions of $L$ then 
    $\varphi_\alpha=\varphi_\beta+\log|g_{\alpha\beta}|$. 
The curvature current $c_1(L,h)\in\mathcal D'_{n-1,n-1}(X)$ of $h$
 is defined by $c_1(L,h)=dd^c\varphi_\alpha$ on 
 $U_\alpha\cap X_{\rm reg}$. We will denote by $h^p$ 
 the singular Hermitian metric induced by $h$ on 
$L^p:=L^{\otimes p}$. If $c_1(L,h)\geq0$ then the weight 
$\varphi_\alpha$
is psh on $U_\alpha\cap X_{\rm reg}$ and since $X$ is normal
it extends to a psh function on $U_\alpha$ \cite[Satz\,4]{GR56}, 
        hence $c_1(L,h)\in\cT(X)$. 
\par Let $L$ be a holomorphic line bundle on a compact normal 
K\"ahler space $(X,\omega)$. Then the space $H^0(X,L)$ of 
holomorphic sections of $L$ is finite dimensional (see e.g.\ 
\cite[Th\'eor\`eme 1,\,p.27]{An:63}). The space $H^0_{(2)}(X,L)$ 
defined as in \eqref{e:bs} is therefore also finite dimensional.

For $p\geq1$, we consider the space $H^0_{(2)}(X,L_p)$ 
defined in \eqref{e:bs}. Recall that $d_p=\dim H^0_{(2)}(X,L_p)$ and 
$S_1^p,\dots,S_{d_p}^p$ is an orthonormal 
basis of $H^0_{(2)}(X,L_p)$. If $x\in X$ and $e_p$ is a
local holomorphic frame of $L_p$ in a neighborhood $U_p$ of $x$ 
we write $S_j^p=s_j^pe_p$, where $s_j^p\in\mathcal O_X(U_p)$. 
Then the Bergman kernel functions and
the Fubini-Study currents of the spaces $H^0_{(2)}(X,L_p)$
are defined as follows:
\begin{equation}\label{e:BFS1}
P_p(x)=\sum_{j=1}^{d_p}|S^p_j(x)|_{h_p}^2\;,\;\;
\gamma_p\vert_{U_p}=\frac{1}{2}\,dd^c
\log\left(\sum_{j=1}^{d_p}|s_j^p|^2\right),
\end{equation}
where $d=\partial+\overline\partial$ and $d^c=\frac{1}{2\pi i}(\partial-\overline\partial)$. 
Note that $P_p,\,\gamma_p$ are independent of the choice 
of basis $S_1^p,\dots,S_{d_p}^p$. 
It follows from \eqref{e:BFS1} that 
$\log P_p\in L^1(X,\omega^n)$ and 
\begin{equation}\label{e:BFS2}
\gamma_p-c_1(L_p,h_p)=\frac{1}{2}\,dd^c\log P_p\,.
\end{equation}
Moreover, as in \cite{CM11,CM13}, one has that 
\begin{equation}\label{e:Bergvar}
P_p(x)=\max\big\{|S(x)|^2_{h_p}:\,S\in H^0_{(2)}(X,L_p),\;
\|S\|_p=1\big\},
\end{equation}
for all $x\in X$ where $|e_p(x)|_{h_p}<\infty$.

We recall that if $S\in H^0(X,L_p)$ the Lelong-Poincar\'e formula shows that
\begin{equation}\label{e:LP}
[S=0]=c_1(L_p,h_p)+dd^c\log|S|_{h_p}\,.
\end{equation}
This follows exactly as in the case when $X$ is smooth (see \cite[Theorem 2.3.3]{MM07}). Indeed, if $X$ is a compact (reduced) analytic space of pure dimension and $S\in H^0(X,L_p)$, the current of integration $[S=0]\in\cT(X)$ is defined as the current with local psh potentials of the form $\log|s|$, where $S=s e_p$, $s\in\mathcal O_X(U_p)$, and $e_p$ is a holomorphic frame of $L_p$ on the open set $U_p\subset X$. If $|e_p|_{h_p}=e^{-\varphi}$, then $\log|S|_{h_p}=\log|s|-\varphi$, which gives \eqref{e:LP}.


\subsection{Special weights of Hermitian metrics on reference covers}\label{SS:refcov} 
Let \((X,\omega)\) be a compact K\"ahler 
manifold of dimension $n$. Let $(U,z)$, $z=(z_1,\ldots,z_n)$, 
be local coordinates centered at a point $x\in X$. For $r>0$ and 
$y\in U$ we denote by
\[\Delta^n(y,r)=\{z\in U: |z_j-y_j|\leq r,\:j=1,\ldots,n\}\]
the (closed) polydisk of polyradius $(r,\ldots,r)$ centered at $y$.
The coordinates $(U,z)$ are called K\"ahler at $y\in U$ if 
\begin{equation}
\omega_z=\frac i2\sum_{j=1}^{n}dz_j\wedge d\overline{z}_j
+O(|z-y|^2)\:\:\text{on \(U\)}.
\end{equation}

\begin{Definition}[{\cite[Definition 2.6]{CMM}}]\label{D:refcov}
A \emph{reference cover} of \(X\) consists of the following data: 
for $j=1,\ldots,N$, a set of points \(x_j\in X\) and 
\begin{enumerate}
  \item  Stein open simply connected coordinate neighborhoods 
  \((U_j,w^{(j)})\) centered at \(x_j\equiv0\),
  \item  \(R_j>0\) such that \(\Delta^n(x_j,2R_j)\Subset U_j\) 
  and for every \(y\in\Delta^n(x_j,2R_j)\) there exist coordinates 
  on \(U_j\) which are K\"ahler at \(y\),
  \item \(X=\bigcup_{j=1}^N\Delta^n(x_j,R_j)\). 
\end{enumerate}
Given the reference cover as above we set \(R=\min R_j\).
\end{Definition}

\par We can construct a reference cover as in \cite[Section 2.5]{CMM}.
On $U_j$ we consider the differential operators $D^\alpha_w$\,, 
$\alpha\in\N^{2n}$, corresponding to the real coordinates 
associated to $w=w^{(j)}$. For a function $\varphi\in\cC^k(U_j)$ 
we set 
\begin{equation}\label{bk:0.1}
\|\varphi\|_{k}=\|\varphi\|_{k,w}
=\sup\big\{|D^\alpha_w\varphi(w)|:\,w\in\Delta^n(x_j,2R_j),
|\alpha|\leq k\big\}.
\end{equation}
Let \((L,h)\) be a Hermitian holomorphic line bundle on \(X\), 
where 
the metric $h$ is of class $\cC^\ell$. Note that \(L|_{U_j}\) is 
trivial. For $k\leq\ell$ set
\begin{equation}\label{bk:0.2}
\begin{split}
\|h\|_{k,U_j}&=\inf\big\{\|\varphi_j\|_k:\,\varphi_j
\in \cC^\ell(U_j)\text{ is a weight of $h$ on $U_j$}\big\},\\
\|h\|_k&=\max\big\{1,\|h\|_{k,U_j}:\,1\leq j\leq N\big\}.
\end{split}
\end{equation}
Recall that \(\varphi_j\) is a weight of \(h\) on \(U_j\) if there
exists a holomorphic frame \(e_j\) of \(L\) on \(U_j\) 
such that \(|e_j|_h=e^{-\varphi_j}\). We have the following:

\begin{Lemma}[{\cite[Lemma 2.7]{CMM}}]\label{L:rc}
There exists a constant \(C>1\) (depending on the reference cover) 
with the following property: 
Given any Hermitian holomorphic line bundle \((L,h)\) on \(X\), where $h$
is of class $\mathscr{C}^3$, any \(j\in\{1,\ldots,N\}\) 
and any \(x\in\Delta^n(x_j,R_j)\) there exist coordinates 
\(z=(z_1,\ldots,z_n)\) on \(\Delta^n(x,R)\) which are centered 
at \(x\equiv0\) and K\"ahler coordinates for \(x\) such that

\par (i) $n!\,dm\leq(1+Cr^2)\omega^n$ and 
$\omega^n\leq(1+Cr^2)n!\,dm$ hold
 on \(\Delta^n(x,r)\) for any \(r<R\) where \(dm=dm(z)\)
  is the Euclidean volume relative to the coordinates $z$\,,
  
\par (ii) \((L,h)\) has a weight \(\varphi\) on \(\Delta^n(x,R)\) 
with \(\varphi(z)=\sum_{j=1}^n\lambda_j|z_j|^2
+\widetilde{\varphi}(z)\), where \(\lambda_j\in\R\) and 
\(|\widetilde{\varphi}(z)|\leq C\|h\|_3|z|^3\) for 
\(z\in\Delta^n(x,R)\). 
\end{Lemma}


\section{Bergman kernel asymptotics}\label{S:berg}
We prove in Section \ref{SS:db} an $L^2$-estimate for the
solution of the $\overline\partial$-equation in the spirit
of Donnelly-Fefferman, which is used in Section \ref{SS:pfTB}
to prove Theorem \ref{T:Bergman}.
  
\subsection{$L^2$-estimates for $\overline{\partial}$}\label{SS:db} 
Let us recall the following version of Demailly's estimates for the $\db$ 
operator \cite[Th\'eor\`eme 5.1]{D82}.

\begin{Theorem}[{\cite[Theorem 2.5]{CMM}}]\label{T:db} Let $Y$, $\dim Y=n$,
be a complete K\"ahler manifold and let $\Omega$ be a K\"ahler
form on $Y$ (not necessarily complete) such that its Ricci form satisfies
$\ric_\Omega\geq-2\pi B\Omega$ on $Y$, for some constant 
    $B>0$. Let $(L_p,h_p)$ be singular Hermitian holomorphic line 
    bundles on $Y$ such that $c_1(L_p,h_p)\geq2a_p\Omega$, 
    where $a_p\to\infty$ as $p\to\infty$, and fix $p_0$ such that 
    $a_p\geq B$ for all $p>p_0$. If $p>p_0$ and 
    $g\in L^2_{0,1}(Y,L_p,loc)$ verifies $\db g=0$ and 
    $\int_Y|g|^2_{h_p}\,\Omega^n<\infty$ then there exists 
    $u\in L^2_{0,0}(Y,L_p,loc)$ such that $\db u=g$ and 
    $\int_Y|u|^2_{h_p}\,\Omega^n\leq\frac{1}{a_p}\,
    \int_Y|g|^2_{h_p}\,\Omega^n$. 
\end{Theorem}

The next result gives a weighted estimate for the solution of the $\db$-equation
which goes back to Donnelly-Fefferman \cite{DoFe83}.
The idea is to twist with a not necessarily plurisubharmonic 
weight whose gradient is however controlled in terms
of its complex Hessian. We follow here \cite[Theorem 4.3]{Berman},
similar estimates were used for $\C^n$ in \cite{Delin,Lin01}. 

\begin{Theorem}\label{T:dbar}
Let $(X,\omega)$ be a compact K\"ahler manifold, $\dim X=n$, 
and $(L_p,h_p)$ be singular Hermitian holomorphic line bundles on $X$ 
such that $h_p$ have locally bounded weights and $c_1(L_p,h_p)\geq a_p\omega\,$, 
where $a_p\to\infty$ as $p\to\infty$. 
Then there exists $p_0\in\N$ with the following property: 
If $v_p$ are real valued functions of class $\cC^2$ on $X$ such that 
\begin{equation}\label{e:vp}
\|\db v_p\|_{L^\infty(X)}\leq\frac{\sqrt{a_p}}{8}\,,\,\;dd^cv_p\geq-\frac{a_p}{2}\,\omega\,,
\end{equation}
then 
\[\int_X|u|_{h_p}^2e^{2v_p}\,\omega^n\leq
\frac{16}{a_p}\,\int_X|\db u|_{h_p}^2e^{2v_p}\,\omega^n\]
holds for $p>p_0$ and for every $\cC^1$-smooth section $u$ of $L_p$ 
which is orthogonal to $H^0(X,L_p)$ with respect to the inner product induced 
by $h_p$ and $\omega^n$.
\end{Theorem}

\begin{proof}
We fix a constant $B>0$ such that $\ric_\omega\geq-2\pi B\omega$ on $X$ and $p_0$ 
such that $a_p\geq4B$ if $p>p_0$. Consider the metric $g_p=h_pe^{-2v_p}$ on $L_p$. 
Then by \eqref{e:vp},
\[c_1(L_p,g_p)=c_1(L_p,h_p)+dd^cv_p\geq\frac{a_p}{2}\,\omega.\]
Moreover
\[\big(e^{2v_p}u,S\big)_{g_p}:=\int_X\langle e^{2v_p}u,S\rangle_{g_p}\,\frac{\omega^n}{n!}=
\int_X\langle u,S\rangle_{h_p}\,\frac{\omega^n}{n!}=0\,,\,\;\forall\,S\in H^0(X,L_p).\]
Let $\alpha=\db\big(e^{2v_p}u\big)=e^{2v_p}(2\db v_p\wedge u+\db u)$. 
By Theorem \ref{T:db} there exists a section 
$\wi u\in L^2_{0,0}(X,L_p)$ such that $\db\wi u=\alpha$ and 
\[\int_X\big|e^{2v_p}u\big|^2_{g_p}\,\omega^n\leq\int_X|\wi u|^2_{g_p}\,\omega^n\leq
\frac{4}{a_p}\,\int_X|\alpha|^2_{g_p}\,\omega^n,\]
where the first inequality follows since $e^{2v_p}u$ is orthogonal to 
$H^0(X,L_p)$ with respect to the inner product $(\cdot,\cdot)_{g_p}$. Using \eqref{e:vp} we obtain  
\[|\alpha|^2_{g_p}=e^{2v_p}|2\db v_p\wedge u+\db u|^2_{h_p}\leq
2e^{2v_p}(4|\db v_p\wedge u|^2_{h_p}+|\db u|^2_{h_p})\leq2e^{2v_p}
\big(\frac{a_p}{16}\,|u|^2_{h_p}+|\db u|^2_{h_p}\big).\] 
It follows that 
\[\int_X|u|^2_{h_p}e^{2v_p}\,\omega^n\leq
\frac{1}{2}\,\int_X|u|^2_{h_p}e^{2v_p}\,\omega^n+
\frac{8}{a_p}\,\int_X|\db u|^2_{h_p}e^{2v_p}\,\omega^n,\]
which implies the conclusion.
\end{proof}

\subsection{Proof of Theorem \ref{T:Bergman}}\label{SS:pfTB}
We recall the following result about the first term asymptotic expansion of the Bergman kernel function $P_p(x)=P_p(x,x)$ (see \eqref{e:BFS1}):

\begin{Theorem}[{\cite[Theorem 1.3]{CMM}}]\label{T:B1}
Let $(X,\omega)$ be a compact K\"ahler manifold of dimension $n$. 
Let $(L_p,h_p)$, $p\geq1$, be a sequence of holomorphic 
line bundles on $X$ with Hermitian metrics $h_p$ of class $\cC^3$ 
whose curvature forms verify \eqref{e:pc} and such that \eqref{e:3d} holds. 
Then there exist $C>0$ depending only on $(X,\omega)$ 
and $p_0\in\mathbb N$ such that  
\begin{equation}\label{e:Bexp00}
\left|P_p(x)\,\frac{\omega^n_x}{c_1(L_p,h_p)^n_x}-1\right|\leq C\varepsilon_p^{2/3}
\end{equation}
holds for every $x\in X$ and $p>p_0$. 
\end{Theorem}

\medskip

Recall that $d(x,y)$, $x,y\in X$, denotes the distance induced by the K\"ahler metric $\omega$. 

\begin{proof}[Proof of Theorem \ref{T:Bergman}]
We use ideas from the proof of \cite[Proposition 9]{Lin01} together 
with methods from \cite[Section 2]{Be03} and \cite[Theorem 1.3]{CMM}. 
Let us consider a reference cover of $X$ as in Definition \ref{D:refcov}. Let $p_0\in\N$ 
be sufficiently large such that 
\[r_p:=a_p^{-1/2}<R/2\] 
and the conclusions of Theorems \ref{T:dbar} and \ref{T:B1} hold for $p>p_0$. 
If $y\in X$ and $r>0$ we let $B(y,r):=\{\zeta\in X:\,d(y,\zeta)<r\}$ 
and we fix a constant $\tau>1$ such that, for every $y\in X$, 
$\Delta^n(y,r_p)\subset B(y,\tau r_p)$, where $\Delta^n(y,r_p)$ is the (closed) 
polydisc centered at $y$ defined using the coordinates centered at $y$ given by Lemma \ref{L:rc}. 

We show first that there exists a constant $C'>1$ with the following property: 
If  $y \in X$, so $y \in \Delta^n(x_j,R_j)$ for some $j$, and $z$ are coordinates 
centered at $y$ as in Lemma \ref{L:rc}, then 
\begin{equation}\label{e:bk1}
|S(y)|^2_{h_p}\leq C'\,\frac{c_1(L_p,h_p)_y^n}{\omega^n_y}\,
\int_{\Delta^n(y,r_p)}|S|^2_{h_p}\,\frac{\omega^n}{n!}\,,
\end{equation}
where $\Delta^n(y,r_p)$ is the (closed) polydisc centered at $y=0$ 
in the coordinates $z$ and $S$ is any continuous section of $L_p$ on $X$ 
which is holomorphic on $\Delta^n(y,r_p)$. Indeed, let 
$$\varphi_p(z) = \varphi^\prime_p(z)+ \widetilde{\varphi}_p (z)\,,
\;\;\varphi^\prime_p(z) = \sum_{l=1}^n \lambda^p_l |z_l|^2\,,$$
be a weight of $h_p$ on $\Delta^n(y,R)$ so that $\widetilde{\varphi}_p$ 
verifies $(ii)$ in Lemma \ref{L:rc} and let
$e_p$ be a frame of $L_p$ on $U_j$ with $|e_p|_{h_p} = e^{-\varphi_p}$. 
Writing $S = se_p$, where 
$s \in \mO(\Delta^n(y,r_p))$, and using the sub-averaging inequality for psh functions we get  
\[|S(y)|^2_{h_p} = |s(0)|^2 \leq \frac{\int_{\Delta^n(0,r_p)} 
|s|^2 e^{-2 \varphi^\prime_p}\,dm}{\int_{\Delta^n(0,r_p)} 
e^{-2 \varphi^\prime_p} \,dm}\,\cdot\]
If $C>1$ is the constant from Lemma \ref{L:rc} then
\begin{eqnarray*}
\int_{\Delta^n(0,r_p)}|s|^2e^{-2\varphi_p'}\,dm&
\leq&(1+Cr^2_p)\exp\!\big(2\max_{\Delta^n(0,r_p)}
\widetilde{\varphi}_p\big)\int_{\Delta^n(0,r_p)}|s|^2
e^{-2\varphi_p}\frac{\omega^n}{n!}\\ 
&\leq&(1+Cr^2_p)\exp\!\big(2C\|h_p\|_3\,r^3_p\big)
\int_{\Delta^n(0,r_p)}|S|^2_{h_p}\,\frac{\omega^n}{n!}\,.
\end{eqnarray*}
Set 
$$E(r):=\int_{|\xi|\leq r}e^{-2|\xi|^2}\,dm(\xi)
=\frac{\pi}{2}\,\left(1-e^{-2r^2}\right),$$ 
where $\,dm$ is the Lebesgue measure on $\C$. Since 
$\lambda_j^p\geq a_p$ and $E(1)>1$ we have 
\[\int_{\Delta^n(0,r_p)}e^{-2\varphi'_p}\,dm\geq 
\frac{E(r_p\sqrt{a_p}\,)^n}{\lambda_1^p\ldots\lambda_n^p}\geq
\frac{1}{\lambda_1^p\ldots\lambda_n^p}\,\cdot\]
Hence 
\[
|S(y)|_{h_p}^2\leq(1+Cr_p^2)\exp\!\big(2C\|h_p\|_3\,r_p^3\big)\,
\lambda_1^p\ldots\lambda_n^p\,\int_{\Delta^n(0,r_p)}|S|^2_{h_p}\,\frac{\omega^n}{n!}\,\cdot
\]
Note that at $y$, $\omega_y=\frac{i}{2}\sum_{j=1}^n dz_j\wedge d\bar{z}_j\,$, 
$c_1(L_p,h_p)_y=dd^c \varphi_p(0)=\frac{i}{\pi}\sum_{j=1}^n\lambda_j^p dz_j\wedge d\bar{z}_j\,$,
thus 
\[\lambda_1^p\ldots\lambda_n^p=
\left(\frac{\pi}{2}\right)^n\,\frac{c_1(L_p,h_p)_y^n}{\omega^n_y}\,\cdot\] 
Since $r_p\to0$ and, by \eqref{e:3d}, $\|h_p\|_3\,r_p^3=\varepsilon_p^3\to0$, 
there exists a constant $C'>1$ such that 
\[\left(\frac{\pi}{2}\right)^n\,(1+Cr_p^2)\exp\!\big(2C\|h_p\|_3\,r_p^3\big)\leq C'\]
for all $p\geq1$. This yields \eqref{e:bk1}.

\medskip

We continue now with the proof of the theorem. Fix $x\in X$. 
There exists a section $S_p=S_{p,x}\in H^0(X,L_p)$ such that
 \[|S_p(y)|_{h_p}^2=|P_p(x,y)|^2_{h_p}\,,\,\;\forall\,y\in X.\]
 Then 
 \[\|S_p\|^2_p=\int_X|S_p(y)|^2_{h_p}\,\frac{\omega_y^n}{n!}
 =\int_X|P_p(x,y)|^2_{h_p}\,\frac{\omega_y^n}{n!}=P_p(x).\]
 By Theorem \ref{T:B1} there exists a constant $C''>1$ such that for all $p\geq1$ and $y\in X$, 
 \begin{equation}\label{e:bk2}
 P_p(y)\leq C''\,\frac{c_1(L_p,h_p)_y^n}{\omega^n_y}\,\cdot
 \end{equation}

\medskip

Assume first that $y\in X$ and $d(x,y)\leq 4\tau r_p=4\tau a_p^{-1/2}$. 
Using \eqref{e:Bergvar} and \eqref{e:bk2} we obtain 
\begin{align*}
|P_p(x,y)|^2_{h_p}&=|S_p(y)|^2_{h_p}\leq P_p(y)\|S_p\|^2_p
=P_p(x)P_p(y)\leq(C'')^2\,\frac{c_1(L_p,h_p)_x^n}{\omega^n_x}\,
\frac{c_1(L_p,h_p)_y^n}{\omega^n_y}\\
&\leq e^{4\tau}(C'')^2\,\frac{c_1(L_p,h_p)_x^n}{\omega^n_x}\,
\frac{c_1(L_p,h_p)_y^n}{\omega^n_y}\,e^{-\sqrt{a_p}\,d(x,y)}\,.
\end{align*}

\medskip

We treat now the case when $y\in X$ and $\delta:=d(x,y)>4\tau r_p=4\tau a_p^{-1/2}$. 
By \eqref{e:bk1} and the definition of $S_p$ we have 
\begin{equation}\label{e:bk3}
|P_p(x,y)|^2_{h_p}=|S_p(y)|^2_{h_p}\leq 
C'\,\frac{c_1(L_p,h_p)_y^n}{\omega^n_y}\,
\int_{\Delta^n(y,r_p)} |P_p(x,\zeta)|^2_{h_p}\,\frac{\omega_\zeta^n}{n!}\,.
\end{equation}
Note that  
\[\Delta^n(x,r_p)\subset B(x,\delta/4)\,,\,\;\Delta^n(y,r_p)\subset 
\{\zeta\in X : d(x,\zeta)>3\delta/4\}.\] 
Let $\chi$ be a non-negative smooth function on $X$ such that  
\begin{equation}\label{e:chi}
\text{$\chi(\zeta)=1$ if $d(x,\zeta)\geq3\delta/4$, \ $\chi(\zeta)=0$ 
if $d(x,\zeta)\leq\delta/2$, \ and 
$|\overline{\partial}\chi(\zeta)|^2\leq \frac{c}{\delta^2}\,\chi(\zeta)$}
\end{equation}
for some constant $c>0$. Then we have
 \begin{align*}
\int_{\Delta^n(y,r_p)} |P_p(x,\zeta)|^2_{h_p}\,\frac{\omega_\zeta^n}{n!}
&\leq\int_X |P_p(x,\zeta)|^2_{h_p}\chi(\zeta)\,\frac{\omega_\zeta^n}{n!}\\
&= \max\left\{|P_p(\chi S)(x)|^2_{h_p}:\,S\in H^0(X,L_p),\,
\int_X|S|^2_{h_p}\chi\,\frac{\omega^n}{n!}=1\right\},
\end{align*}
where 
\[P_p(\chi S)(x)=\int_XP_p(x,\zeta)(\chi(\zeta)S(\zeta))\,\frac{\omega_\zeta^n}{n!}\] 
is the Bergman projection of the smooth section $\chi S$ to $H^0(X,L_p)$.

It remains to estimate $|P_p(\chi S)(x)|^2_{h_p}$, 
where $S\in H^0(X,L_p)$ and $\int_X|S|^2_{h_p}\chi\,\frac{\omega^n}{n!}=1$. 
To this end we consider the smooth section $u$ of $L_p$ given by
\[u:=\chi S-P_p(\chi S).\]
Note that $u$ is orthogonal to $H^0(X,L_p)$ with respect to the inner product 
$(\cdot,\cdot)_p$ induced by $h_p$ and $\omega^n/n!$. 
Moreover, since $\chi(x)=0$, and since $u$ is holomorphic in the polydisc 
$\Delta^n(x,r_p)$ centered at $x$ and defined using the coordinates 
centered at $x$ given by Lemma \ref{L:rc}, it follows by \eqref{e:bk1} that 
\begin{equation}\label{e:bk4}
|P_p(\chi S)(x)|^2_{h_p}=|u(x)|^2_{h_p}\leq C'\,
\frac{c_1(L_p,h_p)_x^n}{\omega^n_x}\,
\int_{\Delta^n(x,r_p)}|u|^2_{h_p}\,\frac{\omega^n}{n!}\,.
\end{equation}
We will estimate the latter integral using Theorem \ref{T:dbar}. 
Let $f:[0,\infty)\to(-\infty,0]$ be a smooth function such that $f(x)=0$ 
for $x\leq1/4$, $f(x)=-x$ for $x\geq1/2$, 
and set $g_\delta(x):=\delta f(x/\delta)$. There exists a constant $M>0$ 
such that $|g'_\delta(x)|\leq M$ 
and $|g''_\delta(x)|\leq M/\delta$ for all $x\geq0$. We define the function
\[v_p(\zeta):=\varepsilon\sqrt{a_p}\,g_\delta(d(x,\zeta))\,,\,\;\zeta\in X.\]
Then there exists a constant $M'>0$ such that 
\[\|\db v_p\|_{L^\infty(X)}\leq M'\varepsilon\sqrt{a_p}\;,\,\;
dd^cv_p\geq-\frac{M'\varepsilon}{\delta}\,\sqrt{a_p}\,
\omega\geq-\frac{M'\varepsilon a_p}{4\tau}\,\omega,\]
since $\delta>4\tau a_p^{-1/2}$. So $v_p$ satisfies \eqref{e:vp} 
if we take $\varepsilon=1/(8M')$. We have that $v_p=0$ in $B(x,\delta/4)\supset\Delta^n(x,r_p)$. 
Moreover  
\[\overline{\partial}u=\overline{\partial}(\chi S)=\db\chi\wedge S\]
is supported in the set $V_\delta:=\{\zeta\in X:\,\delta/2\leq d(x,\zeta)\leq3\delta/4\}$, 
and $v_p(\zeta)=-\varepsilon\sqrt{a_p}\,d(x,\zeta)\leq-\varepsilon\sqrt{a_p}\,\delta/2$ 
on this set. By Theorem \ref{T:dbar} and \eqref{e:chi} we get 
\begin{align*}
\int_{\Delta^n(x,r_p)}|u|^2_{h_p}\,
\frac{\omega^n}{n!}&\leq\int_X|u|_{h_p}^2e^{2v_p}\,\frac{\omega^n}{n!}\leq\frac{16}{a_p}\,
\int_{V_\delta}|\overline{\partial}(\chi S)|_{h_p}^2e^{2v_p}\,\frac{\omega^n}{n!}\\
&\leq\frac{16c}{a_p\delta^2}\,e^{-\varepsilon\sqrt{a_p}\,\delta}
\int_{V_\delta}|S|^2_{h_p}\chi\,\frac{\omega^n}{n!}\leq c\,e^{-\varepsilon\sqrt{a_p}\,\delta},
\end{align*}
since $a_p\delta^2>16\tau^2>16$. Hence \eqref{e:bk4} implies that 
\[|P_p(\chi S)(x)|^2_{h_p}\leq 
C'c\,\frac{c_1(L_p,h_p)_x^n}{\omega^n_x}\,e^{-\varepsilon\sqrt{a_p}\,\delta}\,.\]
It follows that 
\[\int_{\Delta^n(y,r_p)} |P_p(x,\zeta)|^2_{h_p}\,
\frac{\omega_\zeta^n}{n!}\leq C'c\,
\frac{c_1(L_p,h_p)_x^n}{\omega^n_x}\,e^{-\varepsilon\sqrt{a_p}\,d(x,y)}.\]
Combined with \eqref{e:bk3} this gives
\[|P_p(x,y)|^2_{h_p}\leq c(C')^2\,\frac{c_1(L_p,h_p)_x^n}{\omega^n_x}\,
\frac{c_1(L_p,h_p)_y^n}{\omega^n_y}\,e^{-\varepsilon\sqrt{a_p}\,d(x,y)},\]
and the proof is complete.
\end{proof}


\section{Equidistribution for zeros of random holomorphic sections}\label{S:equidist}
In Section \ref{SS:pfT1} we prove Theorem \ref{th1}.  We provide
examples of measures satisfying condition (B) and give applications
of Theorem \ref{th1} in Section \ref{SS:exB}.

\subsection{Proof of Theorem \ref{th1}}\label{SS:pfT1} 
We prove first the following general equidistribution result which combined with \cite[Theorem 1.1]{CMM} 
will yield Theorem 1.1. 

\begin{Theorem}\label{th1gen}
Let $X$ be a compact (reduced) analytic space of pure dimension $n$ and $\omega$ be a Hermitian 
form on $X$. Let $(L_p,h_p)$, $p\geq1$, be singular Hermitian holomorphic line bundles 
on $X$ and let $H^0_{(2)}(X,L_p)$ be the corresponding Bergman spaces defined in \eqref{e:bs} endowed 
with probability measures $\sigma_p$ that verify assumption (B). Let $(\mathcal{H},\sigma)$ be the product 
probability space defined in \eqref{e:calH}. Assume that there exist constants $\alpha_p>0$ such that 
\begin{equation}\label{e:Bka}
\frac{1}{\alpha_p}\,\log P_p\to 0 \,\text{ as $p\to\infty$\,, in $L^1(X,\omega^n)$,}
\end{equation}
where $P_p$ is the Bergman kernel function of $H^0_{(2)}(X,L_p)$ defined in \eqref{e:BFS1}. 
Then the following hold:

\smallskip
\noindent
(i) If $\displaystyle\lim_{p\to\infty}C_p\alpha_p^{-\nu}=0$ 
then $\displaystyle\frac{1}{\alpha_p}\big(\E[s_p=0]-c_1(L_p,h_p)\big)\to 0$\,, 
as $p\to \infty$, in the weak sense of currents on $X$. 

\smallskip
\noindent
(ii) If $\displaystyle\liminf_{p\to\infty}C_p\alpha_p^{-\nu}=0$ 
then there exists a sequence of natural numbers $p_j\nearrow\infty$ 
such that for $\sigma$-a.\,e.\ sequence $\{s_p\}\in\mathcal{H}$ we have
\[\frac{1}{\alpha_{p_j}}\log|s_{p_j}|_{h_{p_j}}\to0\,,\,\;
\frac{1}{\alpha_{p_j}}\big([s_{p_j}=0]-c_1(L_{p_j},h_{p_j})\big) \to0\,,\,
\text{ as $j\to\infty$,}\]
in $L^1(X,\omega^n)$, respectively in the weak sense of currents on $X$.

\smallskip
\noindent
(iii) If $\displaystyle\sum_{p=1}^{\infty}C_p\alpha_p^{-\nu}<\infty$ 
then for $\sigma$-a.\,e.\ sequence
$\{s_p\}\in\mathcal{H}$ we have 
\[\frac{1}{\alpha_p}\log|s_p|_{h_p}\to0\,,\,\;
\frac{1}{\alpha_p}\big([s_p=0]-c_1(L_p,h_p)\big) \to0\,,\,\text{ as $p\to\infty$,}\]
in $L^1(X,\omega^n)$, respectively in the weak sense of currents on $X$.
\end{Theorem}

\begin{proof} Note that if $H^0_{(2)}(X,L_p)\neq\{0\}$ 
then $\log P_p\in L^1(X,\omega^n)$, 
since it is locally the difference of a psh and an integrable function. 
Let $\gamma_p$ be the Fubini-Study currents of the spaces 
$H^0_{(2)}(X,L_p)$ defined in \eqref{e:BFS1}.

$(i)$ Let $\Phi$ be a smooth real valued $(n-1,n-1)$ form on $X$. 
By \eqref{e:BFS2} and 
hypothesis \eqref{e:Bka} we have 
\[\frac{1}{\alpha_p}\,\langle\gamma_p-c_1(L_p,h_p),\Phi\rangle
=\frac{1}{2\alpha_p}\,\int_X\log P_p\,dd^c\Phi\to0\,,\]
so for the first assertion of $(i)$ it suffices to show that 
\begin{equation}\label{e:T11i}
\frac{1}{\alpha_p}\,\langle \E[s_p=0]-\gamma_p,\Phi\rangle\to0\,,\,
\text{ as $p\to\infty$.}
\end{equation}
Note that there exists a constant $c>0$ such that 
for every smooth real valued $(n-1,n-1)$ form $\Phi$ on $X$, 
$$-c\|\Phi\|_{\cC^2}\,\omega^n\leq dd^c\Phi\leq c\|\Phi\|_{\cC^2}\,\omega^n.$$
Hence the total variation of $dd^c\Phi$ satisfies  
$|dd^c\Phi|\leq c\|\Phi\|_{\cC^2}\,\omega^n$.
Indeed, let $\tau:U\hookrightarrow G\subset\C^N$ be a local embedding of $X$, 
where $U\subset X$ and $G\subset\C^N$ are open, 
such that  there exist a smooth real valued $(N-1,N-1)$ form $\wi\Phi$ 
and a Hermitian form $\Omega$ on $G$ with 
$\Phi\mid_{U_{reg}}=\tau^\star\wi\Phi$ and 
$\omega\mid_{U_{reg}}=\tau^\star\Omega$. There exists a constant $c'>0$ 
such that for any smooth real valued $(N-1,N-1)$ form $\varphi$ on $G$ 
and any open set $G_0\Subset G$, we have 
\[-c'\norm\varphi_{\cC^2(G_0)}\Omega^n\leq dd^c\varphi\mid_{G_0} 
\leq c' \norm\varphi_{\cC^2(G_0)}\Omega^n.\]
Our claim follows by taking a finite cover of $X$ 
with sets of form $U_0=\tau^{-1}(G_0)$. 

If $s_p\in H^0_{(2)}(X,L_p)$, using \eqref{e:LP} and \eqref{e:BFS2}, we see that
\begin{equation}\label{eq1}
\big\langle [s_p=0],\Phi\big\rangle = \big\langle c_1(L_p,h_p),\Phi\big\rangle
+ \int_X\log |s_p|_{h_p}dd^c\Phi
= \big\langle \gamma_p,\Phi\big\rangle 
+\int_X\log\frac{|s_p|_{h_p}}{\sqrt{P_p}}\,dd^c\Phi.
\end{equation}
Note that $\log\frac{|s_p|_{h_p}}{\sqrt{P_p}}\in L^1(X,\omega^n)$ 
as it is locally the difference of two psh functions.

We write
$$s_p=\sum_{j=1}^{d_p}a_jS_j^p.$$ 
Moreover, for $x\in X$ we let $e_p$ be a holomorphic frame of $L_p$ 
on a neighborhood $U$ of $x$ and we write
$S_j^p=s_j^pe_p$, where $s_j^p\in\cO_X(U)$. 
Let $\langle a, u^p\rangle =a_1u_1+\ldots +a_{d_p}u_{d_p}$, where 
\begin{equation}\label{e:up}
u^p(x):=(u_1(x),\dots,u_{d_p}(x))\,,\,\; u_j(x)
=\frac{s_j^p(x)}{\sqrt{|s_1^p(x)|^2+\ldots+|s_{d_p}^p(x)|^2}}\,\cdot
\end{equation}
Using H\"older's inequality and assumption (B) it follows that
$$\int_{\hp}\Big|\log\frac{|s_p(x)|_{h_p}}{\sqrt{P_p(x)}}\Big| d\sigma_p(s_p)
=\int_{\Cdp}\big|\log|\langle a ,u^p(x)\rangle|\big| d\sigma_p(a)\leq C_p^{1/\nu}.$$
Hence by Tonelli's theorem 
\[\int_{H^0_{(2)}(X,L_p)}\int_X\Big|\log\frac{|s_p|_{h_p}}{\sqrt{P_p}}\Big|
|dd^c\Phi|\,d\sigma_p(s_p)
\leq C_p^{1/\nu}\int_X|dd^c\Phi|\leq c\,C_p^{1/\nu}\|\Phi\|_{\cC^2}\int_X\omega^n.\]
By (\ref{eq1}) we conclude that 
$$\big\langle \E[s_p=0],\Phi\big\rangle=
\int_{\hp}\big\langle [s_p=0],\Phi\big\rangle\, d\sigma_p(s_p)$$ 
is a well-defined positive closed current which satisfies
$$ \big|\big\langle \E[s_p=0]-\gamma_p,\Phi\big\rangle\big|
\leq c\,C_p^{1/\nu}\|\Phi\|_{\cC^2}\int_X\omega^n.$$
Thus \eqref{e:T11i} holds since $C_p^{1/\nu}/\alpha_p\to0$.

\medskip

For the proof of assertion $(ii)$, since $\liminf_{p\to\infty}C_p\alpha_p^{-\nu}=0$ 
we can find a sequence of natural numbers  $p_j\nearrow\infty$ such that 
$\sum_{j=1}^\infty C_{p_j}\alpha_{p_j}^{-\nu}<\infty$. 
Then we proceed as in the proof of assertion $(iii)$ given below, 
working with $\{p_j\}$ instead of $\{p\}$.

\medskip

$(iii)$ We define 
\[Y_p,\,Z_p:\mathcal H\to[0,\infty)\,,\,\;Y_p(s)
=\frac{1}{\alpha_p}\int_X\big|\log|s_p|_{h_p}\big|\,\omega^n\,,\,\;
Z_p(s)=\frac{1}{\alpha_p}\int_X\Big|\log\frac{|s_p|_{h_p}}{\sqrt{P_p}}\Big|\,\omega^n\,,\]
where $s=\{s_p\}$. So 
\[0\leq Y_p(s)\leq Z_p(s)+m_p\,,\,\text{ where } m_p:=\frac{1}{2\alpha_p}\int_X|\log P_p|\,\omega^n.\]
Hypothesis \eqref{e:Bka} shows that $m_p\to0$ as $p\to\infty$. By H\"older's inequality
\[0\leq Z_p(s)^\nu\leq\frac{1}{\alpha_p^\nu}
\left(\int_X\omega^n\right)^{\nu-1}\int_X\Big|\log\frac{|s_p|_{h_p}}{\sqrt{P_p}}\Big|^\nu\omega^n.\]
For $x\in X$ and $u^p(x)$ as in \eqref{e:up} we obtain using (B) that
\[\int_{\hp}\Big|\log\frac{|s_p(x)|_{h_p}}
{\sqrt{P_p(x)}}\Big|^\nu d\sigma_p(s_p)=\int_{\C^{d_p}}
\big|\log|\langle a ,u^p(x)\rangle|\big|^\nu d\sigma_p(a)\leq C_p.\]
Hence by Tonelli's theorem 
\[\int_{\mathcal H} Z_p(s)^\nu d\sigma(s)\leq\frac{1}{\alpha_p^\nu}
\left(\int_X\omega^n\right)^{\nu-1}\int_X\int_{H^0_{(2)}(X,L_p)}
\Big|\log\frac{|s_p|_{h_p}}{\sqrt{P_p}}\Big|^\nu d\sigma_p(s_p)\;
\omega^n\leq\frac{C_p}{\alpha_p^\nu}\left(\int_X\omega^n\right)^\nu.\]
Therefore
\[\sum_{p=1}^\infty\int_{\mathcal H} Z_p(s)^\nu d\sigma(s)
\leq\left(\int_X\omega^n\right)^\nu\sum_{p=1}^\infty
\frac{C_p}{\alpha_p^\nu}<\infty\,.\]
It follows that $Z_p(s)\to0$, and hence $Y_p(s)\to0$ 
as $p\to\infty$ for $\sigma$-a.\,e.\ $s\in\mathcal{H}$. 
This means that $\frac{1}{\alpha_p}\log|s_p|_{h_p}\to0$ in $L^1(X,\omega^n)$, 
hence by \eqref{e:LP}, 
$\frac{1}{\alpha_p}\big([s_p=0]-c_1(L_p,h_p)\big) \to0$ weakly on $X$,  
for $\sigma$-a.\,e.\ sequence
$\{s_p\}\in\mathcal{H}$. The proof of Theorem \ref{th1gen} is finished.
\end{proof}

\medskip
\noindent
{\bf Proof of Theorem \ref{th1}}. By \cite[Theorem 1.1]{CMM} we have that 
\[\frac{1}{A_p}\,\log P_p\to 0 \,\text{ as $p\to\infty$\,, in $L^1(X,\omega^n)$.}\]
Hence Theorem \ref{th1} follows at once from 
Theorem \ref{th1gen} with $\alpha_p:=A_p$.  \hfill $\Box$

\medskip

Let us give now a variation of Theorem \ref{th1} modeled
on \cite[Corollary 5.6]{CMM}.
It allows to approximate arbitrary $\omega$-psh functions
by logarithms of absolute values
of holomorphic sections.
Let $(X,\omega)$ be a K\"ahler manifold with 
a positive line bundle $(L,h_0)$, where $h_0$ is 
a smooth Hermitian metric such that 
$c_1(L,h_0)=\omega$. The set of singular Hermitian metrics 
$h$ on $L$ with $c_1(L,h)\geq0$ is in one-to-one correspondence 
to the set $PSH(X,\omega)$ of $\omega$-plurisubharmonic 
($\omega$-psh) functions on $X$, by associating to 
$\psi\in PSH(X,\omega)$ the metric 
$h_\psi=h_0e^{-2\psi}$ (see e.g., \cite{D90,GZ05}). 
Note that 
$c_1(L,h_\psi)=\omega+dd^c\psi$.
\begin{Corollary}\label{C:Lp3}
Let $(X,\omega)$ be a compact K\"ahler manifold and $(L,h_0)$ 
be a positive line bundle on $X$ such that $c_1(L,h_0)=\omega$. 
Let $h$ be a singular Hermitian
metric on $L$ with $c_1(L,h)\geq0$ and let $\psi\in PSH(X,\omega)$
be its global weight such that $h=h_0e^{-2\psi}$.
Let $\{n_p\}_{p\geq1}$ be a sequence of natural numbers such 
that 
\begin{equation}\label{n_p}
n_p\to\infty\:\:\text{and $n_p/p\to0$ as $p\to\infty$}.
\end{equation} 
Let  $h_p$ be the metric on $L^p$ given by
\begin{equation}\label{h_p}
h_p=h^{p-n_p}\otimes h_0^{n_p}=h_0^p\,e^{-2(p-n_p)\psi}.
\end{equation}
For $p\geq1$ let $\sigma_p$ be probability measures on 
$\hp=H^0_{(2)}(X,L^p,h_p)$
satisfying condition (B).
Then the following hold:

\smallskip
\noindent
(i) If $\displaystyle\lim_{p\to\infty}C_p\,p^{-\nu}=0$ 
then $\displaystyle\frac{1}{p}\E[s_p=0]\to c_1(L,h)$\,, 
as $p\to \infty$, weakly on $X$. 

\smallskip
\noindent
(ii) If $\displaystyle\liminf_{p\to\infty}C_p\,p^{-\nu}=0$ 
then there exists a sequence of natural numbers $p_j\nearrow\infty$ 
such that for $\sigma$-a.\,e.\ sequence $\{s_p\}\in\mathcal{H}$ we have
as $j\to\infty$,
\[\frac{1}{p_j}\log|s_{p_j}|_{h_0^{p_j}}\to\psi\;\;
\text{in $L^1(X,\omega^n)$}\,,\quad
\frac{1}{p_j}[s_{p_j}=0]\to c_1(L,h)\,,\,
\text{weakly on $X$.}\]

\smallskip
\noindent
(iii) If $\displaystyle\sum_{p=1}^{\infty}C_p\,p^{-\nu}<\infty$ 
then for $\sigma$-a.\,e.\ sequence
$\{s_p\}\in\mathcal{H}$ we have as $p\to\infty$,
\[\frac{1}{p}\log|s_p|_{h_0^p}\to\psi\;\;
\text{in $L^1(X,\omega^n)$}\,,\quad
\frac{1}{p}[s_p=0] \to c_1(L,h)\,,\,\text{weakly on $X$.}\]
\end{Corollary}
\begin{proof}
Note that $\log|s_p|_{h_p}=\log|s_p|_{h_0^p}-(p-n_p)\psi$. 
The corollary follows from Theorem \ref{th1} and the proofs of 
Corollaries 5.2 and 5.6 from \cite{CMM}.
\end{proof}
Corollary \ref{C:Lp3} is an extension of 
\cite[Theorem 5.2]{BL13} which deals with the special case 
when $\psi={\mathcal V}_{K,q}^*$ is the weighted 
$\omega$-psh global extremal function of a compact $K\subset X$.
Note that we use here a different scalar product than in \cite{BL13}.
\smallskip

\begin{Remark}
Let us give a local version of Theorem \ref{th1}. 
Note that when $X$ is smooth any holomorphic line bundle on $X$
is trivial on any contractible Stein open subset $U\subset X$.
Assume that $(X,\omega)$, $(L_p,h_p)$ and $\sigma_p$ 
verify the assumptions (A1), (A2) and (B).
Let $U\subset X$ such that for every $p\geq1$, $L_p|_U$ is trivial and
let $e_p:U\to L_p$ be a holomorphic frame with $|e_p|_{h_p}=e^{-\varphi_p}$,
where $\varphi_p\in PSH(U)$. For a section $s\in H^0(X,L_p)$
write $s=\widetilde{s}\,e_p$\,, with $\widetilde{s}\in\mO(U)$. 
If $\sum_{p=1}^{\infty}C_pA_p^{-\nu}<\infty$,
then for $\sigma$-a.\,e.\ sequence
$\{s_p\}\in\mathcal{H}$ we have as $p\to\infty$,
\[\frac{1}{A_p}\Big(\log|\widetilde{s}_p|-\varphi_p\Big)\to0\:\:
\text{in $L^1(U,\omega^n)$}\,,\quad
\frac{1}{A_p}\Big([\,\widetilde{s}_p=0]-dd^c\varphi_p\Big)\to0\,,\:\:
\text{weakly on $U$.}
\]
In particular, let $(L_p,h_p)=(L^p,h^p)$, where $(L,h)$ is a fixed singular 
Hermitian holomorphic line bundle on $X$ such that 
$c_1(L,h)\geq\varepsilon\omega$ for some $\varepsilon>0$.
Let $U\subset X$ such that $L|_U$ is trivial, let $e:U\to L$
be a holomorphic frame with $|e|_{h}=e^{-\varphi}$,
where $\varphi\in PSH(U)$. Consider the holomorphic
frames $e_p=e^{\otimes p}$ of $L^p|_U$. 
If $\sum_{p=1}^{\infty}C_p\,p^{-\nu}<\infty$,
then for $\sigma$-a.\,e.\ sequence
$\{s_p\}\in\mathcal{H}$ we have as $p\to\infty$,
\[\frac{1}{p}\log|\widetilde{s}_p|\to\varphi\:\:
\text{in $L^1_{loc}(U)$}\,,\quad
\frac{1}{p}[\,\widetilde{s}_p=0]\to dd^c\varphi\,,\:\:
\text{weakly on $U$.}\]
\end{Remark}
\begin{Example}
We formulate now some of the previous results in the case 
of polynomials in $\C^n$. Consider $X={\mathbb P}^n$ and $L_p=\mO(p)$, $p\geq1$, 
where $\mathcal O(1)\to{\mathbb P}^n$ is the hyperplane line bundle.
Let ${\mathbb C}^n\hookrightarrow{\mathbb P}^n$, $\zeta\mapsto[1:\zeta]$, 
be the standard embedding.
The global holomorphic sections $H^0(\mathbb{P}^n,\mO(p))$ 
of $\mO(p)$ are given by homogeneous polynomials 
of degree $p$ in the homogeneous coordinates $z_0,\ldots,z_n$
on $\C^{n+1}$. For any $\alpha\in\N^{n+1}$ the map
$\C^{n+1}\ni z\mapsto z^\alpha$ is identified to a section 
$s_\alpha\in H^0(\mathbb{P}^n,\mO(p))$. 

On $U_0=\{[1:\zeta]\in\mathbb{P}^n:\zeta\in\C^n\}\cong\C^n$
we consider the holomorphic frame $e_p=s_{(p,0,\ldots,0)}$ of 
$\mO(p)$, corresponding to $z_0^p$.
The trivialization of $\mO(p)$ using this frame gives 
an identification
\begin{equation}\label{e:hip}
H^0(\mathbb{P}^n,\mO(p))\to\C_p[\zeta]\,,\:\:s\mapsto s/z_0^p,
\end{equation}
with the space of polynomials of total degree at most $p$,
\[
\C_p[\zeta]=\C_p[\zeta_1,\ldots,\zeta_n]:=
\left\{f\in\C[\zeta_1,\ldots,\zeta_n]:\deg(f)\leq p\right\}.
\]
Let $\omega_{\FS}$ denote the Fubini-Study 
K\"ahler form on $\mathbb{P}^n$ and $h_{\FS}$ be the Fubini-Study metric on $\mO(1)$,
so $c_1(\mO(1),h_{\FS})=\omega_{\FS}$\,. 
The set $PSH(\mathbb{P}^n,p\,\omega_{\FS})$
is in one-to-one correspondence 
to the set $p\mathcal{L}({\mathbb C}^n)$, where $\mathcal{L}({\mathbb C}^n)$ is the
Lelong class of entire psh functions with 
logarithmic  growth (cf.\ \cite[Section 2]{GZ05}): 
\[
\mathcal{L}({\mathbb C}^n)=\left\{\varphi\in PSH(\C^n):\,\exists\, C_\varphi\in\R \text{ such that
$\varphi(z)\leq \log^+\|z\|+C_\varphi$ on $\C^n$}\right\}.
\]
The map $\mathcal L({\mathbb C}^n)
\to PSH({\mathbb P}^n,\omega_{\FS})$
is given by $\varphi\mapsto\wi\varphi$ where 
\[
\wi\varphi(w)=\begin{cases}
\varphi(w)-\frac12\log(1+|w|^2)\,,\quad &w\in\C^n,\\
\limsup\limits_{z\to w,z\in\C^n}\wi\varphi(z)\,,\quad
&w\in\mathbb{P}^n\setminus\C^n.
\end{cases}
\]
The one-to-one correspondence between singular Hermitian
metrics $h_p$ on $\mO(p)$ with $c_1(\mO(p),h_p)\geq0$ and $p\mathcal{L}({\mathbb C}^n)$
is given by sending a metric $h_p$ to its weight $\varphi_p$ on $U_0$ with 
respect to the standard frame $e_p$. Define the $L^2$-space
\[
H^0_{(2)}(\mathbb{P}^n,\mO(p),h_p)=
\left\{s\in H^0(\mathbb{P}^n,\mO(p)):\int_{\mathbb{P}^n}|s|^2_{h_p}
\frac{\omega^n_{\FS}}{n!}<\infty\right\},
\]
with the obvious scalar product. The map \eqref{e:hip} induces an isometry
between this space and the $L^2$-space of polynomials
\begin{equation}\label{e:cp2}
\C_{p,(2)}[\zeta]=
\left\{f\in\C_p[\zeta] :\int_{\C^n}|f|^2e^{-2\varphi_p}
\frac{\omega^n_{\FS}}{n!}<\infty\right\}.
\end{equation}
If $\sigma_p$ are probability measures on $\C_{p,(2)}[\zeta]$
we denote by $\mathcal{H}$ the corresponding product probability space
\((\mathcal{H},\sigma)=
\left(\prod_{p=1}^\infty \C_{p,(2)}[\zeta],\prod_{p=1}^\infty\sigma_p\right) 
\).
\begin{Corollary}\label{C:pol2}
Consider a sequence of functions $\varphi_p\in p\mathcal{L}(\C^n)$
such that $dd^c\varphi_p\geq a_p\,\omega_{\FS}$ on $\C^n$, where
$a_p>0$ and $a_p\to\infty$ as $p\to\infty$. For $p\geq1$ let $\sigma_p$ be probability 
measures on $\C_{p,(2)}[\zeta]$ satisfying condition (B). Assume that $\sum_{p=1}^{\infty}C_pp^{-\nu}<\infty$\,.
Then for $\sigma$-a.\,e.\ sequence
$\{f_p\}\in\mathcal{H}$ we have as $p\to\infty$,
\begin{align*}
&\frac{1}{p}\Big(\log|f_p|-\varphi_p\Big)\to0\:\:
\text{in $L^1(\C^n,\omega_{\FS}^n)\,,\,$ hence in $L^1_{loc}(\C^n)$}\,,\\
&\frac{1}{p}\Big([f_p=0]-dd^c\varphi_p\Big)\to0\,,\:\:\text{weakly on $\C^n$}\,.
\end{align*}
\end{Corollary}

\begin{proof} If $h_p$ is the singular Hermitian metric on $\mO(p)$
corresponding to $\varphi_p$ then 
\[A_p=\int_{{\mathbb P}^n}c_1(\mO(p),h_p)\wedge\omega_{\FS}^{n-1}=p,\text{ and }c_1(\mO(p),h_p)\mid_{\C^n}=dd^c\varphi_p\geq a_p\,\omega_{\FS}.\]
If $T$ denotes the trivial extension of $dd^c\varphi_p$ to ${\mathbb P}^n$ 
then $T\geq a_p\,\omega_{\FS}$ on ${\mathbb P}^n$. 
By Siu's decomposition theorem, $c_1(\mO(p),h_p)=T+b[z_0=0]$, 
where $b\geq0$. 
Hence $c_1(\mO(p),h_p)\geq T\geq a_p\,\omega_{\FS}$ on ${\mathbb P}^n$. 
The corollary now follows directly from Theorem \ref{th1}.
\end{proof}
In particular, we obtain:
\begin{Corollary}\label{C:pol3}
Let $\varphi\in \mathcal{L}(\C^n)$
such that $dd^c\varphi\geq \varepsilon\,\omega_{\FS}$ on $\C^n$
for some constant $\varepsilon>0$. 
For $p\geq1$ construct the spaces $\C_{p,(2)}[\zeta]$ by setting
$\varphi_p=p\varphi$ in \eqref{e:cp2} and let 
$\sigma_p$ be probability measures on $\C_{p,(2)}[\zeta]$
satisfying condition (B). If $\sum_{p=1}^{\infty}C_p\,p^{-\nu}<\infty$,
then for $\sigma$-a.\,e.\ sequence
$\{f_p\}\in\mathcal{H}$ we have as $p\to\infty$,
\begin{equation}\label{e:pp}
\frac{1}{p}\log|f_p|\to\varphi\:\:
\text{in $L^1(\C^n,\omega_{\FS}^n)$\,,}\quad 
\frac{1}{p}[f_p=0]\to dd^c\varphi\,,\:\:\text{weakly on $\C^n$\,.}
\end{equation}
\end{Corollary}
We can also apply Corollary \ref{C:Lp3} to the setting of polynomials in $\C^n$
and obtain a version of Corollary \ref{C:pol3} for arbitrary $\varphi\in \mathcal{L}(\C^n)$.
\begin{Corollary}\label{C:pol4}
Let $\varphi\in \mathcal{L}(\C^n)$ and let $h$ be the singular Hermitian metric 
on $\mO(1)$ corresponding to $\varphi$. 
Let $\{n_p\}_{p\geq1}$ 
be a sequence of natural numbers such 
that \eqref{n_p} is satisfied.
Consider the metric $h_p$ on $\mO(p)$ given by
$h_p=h^{p-n_p}\otimes h_{\FS}^{n_p}$ (cf.\ \eqref{h_p}).
For $p\geq1$ let $\sigma_p$ be probability measures on 
$H^0_{(2)}(\mathbb{P}^n,\mO(p),h_p)\cong \C_{p,(2)}[\zeta]$ 
satisfying condition (B). 
If $\sum_{p=1}^{\infty}C_p\,p^{-\nu}<\infty$,
then for $\sigma$-a.\,e.\ sequence
$\{f_p\}\in\mathcal{H}$ we have \eqref{e:pp} as $p\to\infty$.
\end{Corollary}
This is an extension (with a different scalar product) of 
\cite[Theorem 4.2]{BL13} which deals with the special case 
when $\psi=V_{K,Q}^*$ is the 
weighted pluricomplex Green function
of a nonpluripolar compact $K\subset \C^n$ \cite[(3.2)]{BL13}.
\end{Example}

\subsection{Classes of measures verifying assumption (B)}\label{SS:exB}

In this section we give important examples of measures 
that verify condition (B) and we specialize Theorem \ref{th1} to these measures.  

\subsubsection{Gaussians}\label{s:Gauss}
We consider here the measures $\sigma_k$ on $\C^k$ that have Gaussian density,
\begin{equation}\label{e:Gauss}
d\sigma_k(a)=\frac{1}{\pi^k}\,e^{-\|a\|^2}\,dV_k(a)\,,
\end{equation}
where $a=(a_1,\ldots,a_k)\in\C^k$ and $V_k$ is the Lebesgue measure on $\C^k$.

\begin{Lemma}\label{L:Gauss} For every integer $k\geq1$ and every $\nu\geq1$,
\[\int_{\C^k}|\log|\langle a,u\rangle||^\nu\,d\sigma_k(a)
=\Gamma_\nu:=2\int_0^\infty r|\log r|^\nu e^{-r^2}\,dr\,,\,\;
\forall\,u\in\C^k,\;\|u\|=1\,.\]
\end{Lemma}

\begin{proof}
Since $\sigma_k$ is unitary invariant we have
\[\int_{\C^k}|\log|\langle a,u\rangle||^\nu\,d\sigma_k(a)
=\int_{\C^k}|\log|a_1||^\nu\,d\sigma_k(a)=\frac{1}{\pi}\,\int_\C
|\log|a_1||^\nu e^{-|a_1|^2}\,dV_1(a_1)\,.\]
\end{proof}

Lemma \ref{L:Gauss} implies at once that in this case 
Theorem \ref{th1} takes the following simpler form:

\begin{Theorem}\label{T:Gauss}
Assume that $(X,\omega)$, $(L_p,h_p)$ verify the assumptions (A1), (A2), 
and $\sigma_p:=\sigma_{d_p}$ is the measure given by \eqref{e:Gauss} 
on $H^0_{(2)}(X,L_p)\simeq \C^{d_p}$. 
Then the following hold:

\smallskip
\noindent
(i) $\displaystyle\frac{1}{A_p}\big(\E[s_p=0]-c_1(L_p,h_p)\big)\to 0$\,, 
as $p\to \infty$, in the weak sense of currents on $X$. 
Moreover, there exists a sequence $p_j\nearrow\infty$ such that for 
$\sigma$-a.\,e.\ sequence $\{s_p\}\in\mathcal{H}$ we have
\[\frac{1}{A_{p_j}}\log|s_{p_j}|_{h_{p_j}}\to0\,,\,\;
\frac{1}{A_{p_j}}\big([s_{p_j}=0]-c_1(L_{p_j},h_{p_j})\big) \to0\,,\,
\text{ as $j\to\infty$,}\]
in $L^1(X,\omega^n)$, respectively in the weak sense of currents on $X$.

\smallskip
\noindent
(ii) If $\displaystyle\sum_{p=1}^{\infty}A_p^{-\nu}<\infty$ 
for some $\nu\geq1$, then for $\sigma$-a.\,e.\ sequence
$\{s_p\}\in\mathcal{H}$ we have 
\[\frac{1}{A_p}\log|s_p|_{h_p}\to0\,,\,\;
\frac{1}{A_p}\big([s_p=0]-c_1(L_p,h_p)\big) \to0\,,\,\text{ as $p\to\infty$,}\]
in $L^1(X,\omega^n)$, respectively in the weak sense of currents on $X$.
\end{Theorem}

\subsubsection{Fubini-Study volumes}\label{s:FSvol}
The Fubini-Study volume on the projective space ${\mathbb P}^k\supset\C^k$ 
is given by the measure $\sigma_k$ on $\C^k$ with density 
\begin{equation}\label{e:FSvol}
d\sigma_k(a)=\frac{k!}{\pi^k}\,\frac{1}{(1+\|a\|^2)^{k+1}}\,dV_k(a)\,,
\end{equation}
where $a=(a_1,\ldots,a_k)\in\C^k$ and $V_k$ is the Lebesgue measure on $\C^k$.

\begin{Lemma}\label{L:FSvol} For every integer $k\geq1$ and every $\nu\geq1$,
\[\int_{\C^k}|\log|\langle a,u\rangle||^\nu\,d\sigma_k(a)
=\Gamma_\nu:=2\int_0^\infty\frac{r|\log r|^\nu}{(1+r^2)^2}\,dr\,,\,\;
\forall\,u\in\C^k,\;\|u\|=1\,.\]
\end{Lemma}

\begin{proof}
Recall that the area of the unit sphere in $\C^k$ is $s_{2k}=2\pi^k/(k-1)!$. 
Since $\sigma_k$ is unitary invariant we have
\[\int_{\C^k}|\log|\langle a,u\rangle||^\nu\,d\sigma_k(a)
=\int_{\C^k}|\log|a_1||^\nu\,d\sigma_k(a)=
4k(k-1)\int_0^\infty\int_0^\infty\frac{r|\log r|^\nu\rho^{2k-3}}{(1+r^2+\rho^2)^{k+1}}\,d\rho\,dr,\]
where we used polar coordinates for $a_1$ and spherical coordinates for $(a_2,\ldots,a_k)\in\C^{k-1}$. 
Changing variables $\rho^2=(1+r^2)x(1-x)^{-1}$, $2\rho\,d\rho
=(1+r^2)(1-x)^{-2}\,dx$, in the inner integral we obtain
\[\int_0^\infty\frac{\rho^{2k-3}}{(1+r^2+\rho^2)^{k+1}}\,d\rho
=\frac{1}{2(1+r^2)^2}\,\int_0^1x^{k-2}(1-x)\,dx=\frac{1}{2k(k-1)(1+r^2)^2}\,,\]
and the lemma follows.
\end{proof}

Lemma \ref{L:FSvol} shows that the conclusions of Theorem \ref{T:Gauss} 
hold for the measures $\sigma_p:=\sigma_{d_p}$ given by \eqref{e:FSvol} 
on $H^0_{(2)}(X,L_p)\simeq \C^{d_p}$. 

\medskip

More generally, one can consider radial probability measures on $\C^k$ with density
\begin{equation}\label{e:genrad}
d\sigma_{k,\alpha}(a)=\frac{\Gamma(k+\alpha)}{\Gamma(\alpha)\pi^k}\,\frac{1}{(1+\|a\|^2)^{k+\alpha}}\,dV_k(a)\,,
\end{equation}
where $\alpha>0$ and $\Gamma$ is the Gamma function. As in the proof of Lemma \ref{L:FSvol} one 
can show that for every integer $k\geq1$ and every $\nu\geq1$,
\[\int_{\C^k}|\log|\langle a,u\rangle||^\nu\,d\sigma_{k,\alpha}(a)
=\Gamma_{\nu,\alpha}:=2\alpha\int_0^\infty\frac{r|\log r|^\nu}{(1+r^2)^{1+\alpha}}\,dr\,,\,\;
\forall\,u\in\C^k,\;\|u\|=1\,.\]

\subsubsection{Area measure of spheres}\label{s:sphere}
Let ${\mathcal A}_k$ be the surface measure on the unit sphere 
${\mathbf S}^{2k-1}$ in ${\mathbb C}^k$, 
so ${\mathcal A}_k\big({\mathbf S}^{2k-1}\big)=2\pi^k/(k-1)!$, and let 

\begin{equation}\label{e:sphere}
\sigma_k=\frac{1}{{\mathcal A}_k\big({\mathbf S}^{2k-1}\big)}\,{\mathcal A}_k\,.
\end{equation}

\begin{Lemma}\label{L:sphere} If $\nu\geq1$ there exists a constant 
$M_\nu>0$ such that for every integer $k\geq2$,
\[\int_{{\mathbf S}^{2k-1}}|\log|\langle a,u\rangle||^\nu\,d\sigma_k(a)
\leq M_\nu\,(\log k)^\nu\,,\,\;\forall\,u\in\C^k,\;\|u\|=1\,.\]
\end{Lemma}

\begin{proof}  We use spherical coordinates 
$(\theta_1,\ldots,\theta_{2k-2},\varphi)\in\left[-\frac{\pi}{2}, 
\frac{\pi}{2}\right]^{2k-2}\times[0,2\pi]$ 
on ${\mathbf S}^{2k-1}$ such that 
$$a_k=\sin\theta_{2k-3}\cos\theta_{2k-2}+i\sin\theta_{2k-2}\,,\;\;
d{\mathcal A}_k=\cos\theta_1\cos^2\theta_2\ldots\cos^{2k-2}\theta_{2k-2}\;
d\theta_1\ldots d\theta_{2k-2}d\varphi\,.$$
Since $\sigma_k$ is unitary invariant we argue in the proof of \cite[Lemma 4.3]{CMM} 
and obtain that there exists a constant $c>0$ such that for every $k$ and $\nu$, 
\begin{align*}
\int_{{\mathbf S}^{2k-1}}|\log|\langle a,u\rangle||^\nu\,d\sigma_k(a)
&=\int_{{\mathbf S}^{2k-1}}|\log|a_k||^\nu\,d\sigma_k(a)\\
&\leq\frac{ck}{2^\nu}\int_0^1\int_0^1(1-x^2)^{k-3/2}(1-y^2)^{k-2}|
\log(x^2+y^2-x^2y^2)|^\nu\,dxdy\\
&\leq\frac{\pi ck}{2^{\nu+1}}\int_0^1(1-t)^{k-2}|\log t|^\nu\,dt\;.
\end{align*}
Note that 
\[f(t):=t^{1/2}|\log t|^\nu\leq f\big(e^{-2\nu}\big)=(2\nu/e)^\nu\,,\, \text{ for $0<t\leq1$.}\]
It follows that 
\begin{align*}
\int_0^1(1-t)^{k-2}|\log t|^\nu\,dt
&\leq\left(\frac{2\nu}{e}\right)^\nu\int_0^{1/k^2}(1-t)^{k-2}t^{-1/2}\,dt
+\int_{1/k^2}^1(1-t)^{k-2}|\log t|^\nu\,dt\\
&\leq\left(\frac{2\nu}{e}\right)^\nu\int_0^{1/k^2}t^{-1/2}\,dt
+2^\nu(\log k)^\nu\int_{1/k^2}^1(1-t)^{k-2}\,dt\\
&\leq\left(\frac{2\nu}{e}\right)^\nu\frac{2}{k}+\frac{2^\nu(\log k)^\nu}{k-1}\,,
\end{align*}
which implies the conclusion of the lemma.
\end{proof}

Lemma \ref{L:sphere} implies that in this case 
Theorem \ref{th1} takes the following simpler form:

\begin{Theorem}\label{T:sphere}
Assume that $(X,\omega)$, $(L_p,h_p)$ verify the assumptions (A1), (A2), 
and $\sigma_p:=\sigma_{d_p}$ is the measure given by \eqref{e:sphere} 
on the unit sphere of $H^0_{(2)}(X,L_p)\simeq \C^{d_p}$. 
Then the following hold:

\smallskip
\noindent
(i) If $\displaystyle\lim_{p\to\infty}\frac{\log d_p}{A_p}=0$ then 
$\displaystyle\frac{1}{A_p}\big(\E[s_p=0]-c_1(L_p,h_p)\big)\to 0$\,, 
as $p\to \infty$, in the weak sense of currents on $X$. 

\smallskip
\noindent
(ii) If $\displaystyle\liminf_{p\to\infty}\frac{\log d_p}{A_p}=0$ 
then there exists a sequence $p_j\nearrow\infty$ such that for $\sigma$-a.\,e.\ sequence 
$\{s_p\}\in\mathcal{H}$ we have
\[\frac{1}{A_{p_j}}\log|s_{p_j}|_{h_{p_j}}\to0\,,\,\;
\frac{1}{A_{p_j}}\big([s_{p_j}=0]-c_1(L_{p_j},h_{p_j})\big) \to0\,,\,\text{ as $j\to\infty$,}\]
in $L^1(X,\omega^n)$, respectively in the weak sense of currents on $X$.

\smallskip
\noindent
(iii) If $\displaystyle\sum_{p=1}^{\infty}\left(\frac{\log d_p}{A_p}\right)^\nu<\infty$ 
for some $\nu\geq1$, then for $\sigma$-a.\,e.\ sequence
$\{s_p\}\in\mathcal{H}$ we have 
\[\frac{1}{A_p}\log|s_p|_{h_p}\to0\,,\,\;
\frac{1}{A_p}\big([s_p=0]-c_1(L_p,h_p)\big) \to0\,,\,\text{ as $p\to\infty$,}\]
in $L^1(X,\omega^n)$, respectively in the weak sense of currents on $X$.
\end{Theorem}

We remark that the assertion $(ii)$ of Theorem \ref{T:sphere} 
was proved in \cite[Theorem 4.2]{CMM}. That paper also gives two general 
examples of sequences of line bundles $L_p$ for which  
\[\lim_{p\to0}\frac{\log\dim H^0(X,L_p)}{A_p}=0,\]
see \cite[Proposition 4.4]{CMM} and \cite[Proposition 4.5]{CMM}. 
In particular, if $X$ is smooth and each $L_p$ is semiample then 
it is shown in  \cite[Proposition 4.5]{CMM} that 
\[\dim H^0(X,L_p)=O(A_p^N).\] 
Therefore $\displaystyle\lim_{p\to\infty}(\log d_p)/A_p=0$. Moreover, 
since $\log d_p<\sqrt{A_p}$ for $p$ sufficiently large, the hypothesis that 
$\sum_{p=1}^{\infty}\left(\frac{\log d_p}{A_p}\right)^\nu<\infty$, 
for some $\nu\geq1$, in Theorem \ref{T:sphere} $(iii)$, can be replaced by the condition 
that $\sum_{p=1}^{\infty}A_p^{-\nu}<\infty$ for some $\nu\geq1$.

\begin{Remark}
We note that for unitary invariant measures $\sigma_p$, 
like those from Sections \ref{s:Gauss}-\ref{s:sphere}, 
the probability space $(\hp,\sigma_p)$ does not depend 
on the choice of orthonormal basis. Other important classes 
of probability measures which do not depend on the choice of 
orthonormal basis and are not unitary invariant are given in \cite{FZ11} 
(see formulas (5), (6) and (7) therein). These measures $\gamma_N$ 
are easily seen to be dominated by measures $\sigma_N$ on the space 
$\mathcal P_N\simeq{\mathbb C}^{N+1}$ of polynomials in $\mathbb C$ 
of degree at most $N$, with Gaussian type density of the form 
\[d\sigma_N(a)=e^{C-\varepsilon\|a\|^2}\,dV_{N+1}(a)\,.\]
Indeed, the polynomial $P(x)$ from \cite[(7)]{FZ11} is bounded from 
below on $[0,+\infty)$, hence $P(x)\geq\varepsilon x-C$ for all $x\geq0$, 
with some constants $\varepsilon,C>0$. An argument analogous to that 
in the proof of Lemma \ref{L:Gauss} shows that the measures $\gamma_N$ 
verify assumption (B) for every $\nu\geq1$ with constants $C_N=\Gamma_\nu$ 
independent of $N$. In particular, if the metric $h$ and the measure $\nu$ 
in the definition of $\gamma_N$ \cite[(5)]{FZ11} is positively curved, 
respectively a K\"ahler form on ${\mathbb P}^1$, then our Theorem \ref{th1} holds in 
the setting of \cite{FZ11} for the measures $\gamma_N$.
\end{Remark}

\subsubsection{Measures with heavy tail and small ball probability}\label{s:tail}
Let $\sigma_p$ be probability measures on $\hp\simeq \C^{d_p}$ verifying the following: 
There exist a constant $\rho>1$ and for every $p\geq1$ constants $C'_p>0$ such that:
\begin{itemize}
\item[(B1)] For all $R\geq1$ the \textit{tail probability} satisfies 
\[\sigma_p\big(\{a\in \C^{d_p}:\log\|a\|>R\}\big)\leq \frac{C'_p}{R^\rho}\,;\] 
\item[(B2)] For all $R\geq1$ and for each unit vector $u\in\Cdp$, 
the \textit{small ball probability} satisfies
\[\sigma_p\big(\{a\in\Cdp:\log |\langle a,u\rangle|<-R\}\big)\leq \frac{C'_p}{R^\rho}\,\cdot\]
\end{itemize}

\begin{Lemma}\label{lem1}
If $\sigma_p$ are probability measures on $\C^{d_p}$ verifying (B1) and (B2) 
with some constant $\rho>1$, then $\sigma_p$ verify (B) for any constant $1\leq\nu<\rho$.
\end{Lemma}

\begin{proof} Let $\nu<\rho$ and $u\in\C^{d_p}$ be a unit vector. By (B1), (B2) we have
\[\sigma_p\big(\{a\in\Cdp: |\log|\langle a,u\rangle||>R\}\big)
\leq \frac{2C'_p}{R^\rho}\,,\,\;\forall\,R\geq1\,.\]
Hence
\begin{align*}
\int_{\C^{d_p}}|\log|\langle a,u\rangle||^\nu\,d\sigma_p(a)&=
\nu\int_0^\infty R^{\nu-1}\sigma_p\big(\{a\in\Cdp: |\log|\langle a,u\rangle||>R\}\big)\,dR\\
&\leq\nu\int_0^1 R^{\nu-1}\,dR+2\nu C'_p\int_1^\infty R^{\nu-\rho-1}\,dR
=1+\frac{2\nu C'_p}{\rho-\nu}=:C_p\,.
\end{align*}
\end{proof}

\subsubsection{Random holomorphic sections with i.i.d.\ coefficients}\label{iidcase}
Next, we consider random linear combinations of the orthonormal basis 
$(S_j^p)_{j=1}^{d_p}$ with independent identically distributed (i.i.d.)\ coefficients. 
More precisely, let $\{a_j^p\}_{j=1}^{d_p}$ be an array 
of i.i.d.\ complex random variables whose distribution law is denoted by $\boldsymbol{P}$. 
Then a random holomorphic section is of the form
$$s_p=\sum_{j=1}^{d_p}a_j^pS_j^p.$$ 
We endow the space $H^0_{(2)}(X,L_p)$ with the $d_p$-fold product measure 
$\sigma_p$ induced by $\boldsymbol{P}$.

\begin{Lemma}\label{iid} Assume that $a_j^p$ are i.i.d.\ complex valued 
random variables whose distribution law $\boldsymbol{P}$ has density $\phi$, 
such that $\phi:\C\to[0,M]$ is a bounded function and there exist $c>0,\,\rho>1$ with 
\begin{equation}\label{tailc} 
\boldsymbol{P}(\{z\in\C: \log|z|>R\})\leq \frac{c}{R^{\rho}} \,,\,\;\forall\,R\geq1.
\end{equation}
Then the product measures $\sigma_p$ on $\Cdp$ satisfy condition (B) 
for any $1\leq \nu<\rho$, with constants  
$C_p=\Gamma d_p^{\nu/\rho}$, where $\Gamma=\Gamma(M,c,\rho,\nu)>0$. 
In particular, if $d_p=O(A_p^N)$ for some $N\in\N$ and $\rho>N$, 
then $\sigma_p$ satisfy condition (B) for any $1\leq \nu<\rho$ with $C_p=O(A_p^{N\nu/\rho})=o(A_p^\nu)$.
\end{Lemma}

\begin{proof} Let $u=(u_1,\ldots,u_{d_p})\in\C^{d_p}$ be a unit vector. For $R\geq\log d_p$ we have 
\[\{a\in\Cdp: \log|\langle a,u\rangle|>R\}\subset \bigcup_{j=1}^{d_p}\big\{a_j:|a_j|>e^{R-\frac12\log d_p}\big\},\]
so by (\ref{tailc}),
\begin{equation}\label{ues}
\sigma_p\big(\{a\in\Cdp: \log|\langle a,u\rangle|>R\}\big)\leq 
d_p\,\boldsymbol{P}\big(\{a_j^p\in\mathbb{C}:|a_j^p|>e^{R-\frac12\log d_p}\}\big)\leq 
\frac{2^\rho cd_p}{R^\rho}\,\cdot
\end{equation}

On the other hand, we have $|u_j|\geq d_p^{-1/2}$ for some $j\in \{1,\dots ,d_p\}$. 
We may assume $j=1$ for simplicity and apply the change of variables
$$\alpha_1=\sum_{j=1}^{d_p}a_j^pu_j,\ \alpha_2=a_2^p,\ \dots, \alpha_{d_p}=a_{d_p}^p. $$
Then, using the assumption $\phi\leq M$,  
\begin{equation}\label{les} 
\begin{split}
\sigma_p\big(\{a\in&\Cdp: \log|\langle a,u\rangle|<-R\}\big)\\ & =
\int_{\mathbb{C}^{d_p-1}}\int_{|\alpha_1|<e^{-R}} 
\phi\left(\frac{\alpha_1-\sum_{j=2}^{d_p}\alpha_ju_j}{u_1}\right)\phi(\alpha_2)\dots
\phi(\alpha_{d_p})\,\frac{d\alpha_1\dots d\alpha_{d_p}}{|u_1|^2}\\
&\leq M\pi d_pe^{-2R}.
\end{split}
\end{equation} 
For $R_0\geq\log d_p$ we obtain using (\ref{ues}) and (\ref{les})
\begin{align*}
\int_{\C^{d_p}}|\log|\langle a,u\rangle||^\nu&\,d\sigma_p(a)=
\nu\int_0^\infty R^{\nu-1}\sigma_p\big(\{a\in\Cdp: |\log|\langle a,u\rangle||>R\}\big)\,dR\\
&\leq\nu\int_0^{R_0} R^{\nu-1}dR + \nu\int_{R_0}^{\infty} 
R^{\nu-1}\sigma_p\big(\{a\in\Cdp: |\log|\langle a,u\rangle||>R\}\big)\,dR\\
&\leq R_0^\nu + \nu\int_{R_0}^{\infty} R^{\nu-1}\left(\frac{2^\rho cd_p}{R^\rho}+
M\pi d_pe^{-2R}\right)dR\,.
\end{align*}
Since $R^{\nu-1}e^{-R}\leq((\nu-1)/e)^{\nu-1}$ for $R>0$, and since $R_0\geq\log d_p$, we get 
\begin{align*}
\int_{\C^{d_p}}|\log|\langle a,u\rangle||^\nu\,d\sigma_p(a)&\leq R_0^\nu+
\frac{2^\rho\nu cd_pR_0^{\nu-\rho}}{\rho-\nu}+
M\pi\nu d_p\left(\frac{\nu-1}{e}\right)^{\nu-1}\int_{R_0}^\infty e^{-R}dR\\
&\leq R_0^\nu\left(1+\frac{2^\rho\nu cd_p}{(\rho-\nu)R_0^\rho}\right)+
M\pi\nu \left(\frac{\nu-1}{e}\right)^{\nu-1}.
\end{align*}
Choosing $R_0^\rho=d_p$ this implies that 
\[\int_{\C^{d_p}}|\log|\langle a,u\rangle||^\nu\,d\sigma_p(a)\leq\Gamma d_p^{\nu/\rho},\]
where $\Gamma>0$ is a constant that depends on $M$, $c$, $\rho$ and $\nu$. 
\end{proof}

We remark that  if $X$ is smooth and each $L_p$ is semiample then 
$d_p=O(A_p^N)$ (see \cite[Proposition 4.5]{CMM}) and Lemma \ref{iid} applies.

\subsubsection{Locally moderate measures}\label{moderatecase}
Let $X$ be a complex manifold and $\sigma$ be a positive measure on $X$. 
Following \cite{DNS}, we say that $\sigma$ is locally moderate if for any open set 
$U\subset X$, any compact set $K\subset U$, and any compact family 
$\mathscr{F}$ of psh functions on $U$, there exist constants $c,\alpha>0$ such that
\begin{equation}\label{mode1}
\int_Ke^{-\alpha\psi}d\sigma\leq c\,,\,\; \forall\,\psi\in\mathscr{F}.
\end{equation}  
Note that a locally moderate measure $\sigma$ does not put any mass on pluripolar sets. 
The existence of $c,\alpha$ in (\ref{mode1}) is equivalent to existence of $c',\alpha'>0$ satisfying
\[\sigma(\{z\in K: \psi(z)<-t\})\leq c'e^{-\alpha't}\,,\]
for any $t\geq0$ and $\psi\in\mathscr{F}$. Important examples are provided by 
the Monge-Amp\`ere measures of H\"older continuous psh functions \cite[Theorem 1.1, Corollary 1.2]{DNS}. 

\begin{Lemma}\label{moderate}
If $\sigma_p\,$, $p\geq1$, is a locally moderate probability measure with compact support in 
${\mathbb C}^{d_p}\simeq H^0_{(2)}(X,L_p)$, then $\sigma_p$ satisfies condition $(B)$ for every $\nu\geq1$. 
\end{Lemma}

\begin{proof}
Consider the compact family of psh functions 
$\mathscr{F}=\{\psi_u: u\in\mathbf{S}^{2d_p-1}\}$, 
where $\psi_u:\Cdp\to[-\infty,\infty)$, $ \psi_u(a)=\log|\langle a,u\rangle|$. Let $R_p\geq1$ be 
such that $\|a\|\leq R_p$ for all $a\in\supp\sigma_p$. Then 
\[|\psi_u(a)|=-\psi_u(a)+\max\{0,2\psi_u(a)\}\leq-\psi_u(a)+2\log R_p\]
holds for all $a\in\supp\sigma_p$ and $\psi_u\in\mathcal{F}$.
Since $\sigma_p$ is locally moderate and with compact support, 
there exist constants $c_p,\alpha_p>0$ such that \eqref{mode1} 
holds for every $\psi_u\in\mathcal{F}$ and with the integral over ${\mathbb C}^{d_p}$. Fix $\nu\geq1$. 
As $x^\nu\leq c'e^{\alpha_p x}$ for all $x\geq0$, 
with some constant $c'>0$ depending on $p,\nu$, we conclude that
\[\int_{\C^{d_p}}|\psi_u(a)|^\nu\,d\sigma_p(a)\leq c'\int_{{\mathbb C}^{d_p}}e^{\alpha_p|\psi_u(a)|}\,d\sigma_p(a)\leq
c'R_p^{2\alpha_p}\int_{{\mathbb C}^{d_p}}e^{-\alpha_p\psi_u(a)}\,d\sigma_p(a)\leq c'c_pR_p^{2\alpha_p}\,.\]
\end{proof}



\begin{thebibliography}{XXXXX}
\bibitem[A]{An:63} A.~Andreotti, 
\emph{{Th\'eor\`emes de d\'ependance alg\'ebrique sur 
les espaces complexes pseudo-concaves}}, 
Bull.\ Soc.\ Math.\ France \textbf{91} (1963), 1--38.

\bibitem[Ba1]{B6}
T.~Bayraktar, 
{\em Equidistribution of zeros of random holomorphic sections}, 
Indiana Univ. Math. J. {\bf 65} (2016), 1759--1793.

\bibitem[Ba2]{B9}
T.~Bayraktar, 
{\em Asymptotic normality of linear statistics of zeros of random polynomials}, 
Proc. Amer. Math. Soc. {\bf 145} (2017), 2917--2929.

\bibitem[Ba3]{B7}
T.~Bayraktar, 
{\em Zero distribution of random sparse polynomials}, 
Michigan Math. J. {\bf 66} (2017), 389--419.

\bibitem[Be]{Be03} B. Berndtsson, 
\emph{Bergman kernels related to Hermitian line bundles over compact complex manifolds}, 
Explorations in complex and Riemannian geometry, Contemp. Math. {\bf 332} (2003), 1--17.

\bibitem[Ber]{Berman}
R.~J. Berman,
\emph{Determinantal point processes and fermions on complex manifolds: bulk universality},
preprint 2008, arXiv:0811.3341.

\bibitem[BSZ]{BSZ}
P.~Bleher, B.~Shiffman, and S.~Zelditch,
{\em Universality and scaling of correlations between zeros on complex manifolds},
Invent. Math. 1{\bf 42} (2000), 351--395.

\bibitem[BL]{BL13} T. Bloom and N. Levenberg, 
{\em Random polynomials and pluripotential-theoretic extremal functions}, 
Potential Anal. {\bf 42} (2015), 311--334.

\bibitem[Ch1]{Christ91} M.\ Christ,
\emph{On the $\overline{\partial}$ equation in weighted 
$L^2$-norms in $\C ^1$}, J. Geom. Anal. \textbf{3} (1991), 193--230.

\bibitem[Ch2]{Christ13} M.\ Christ,
\emph{Upper bounds for Bergman kernels associated to positive
line bundles with smooth Hermitian metrics},  arXiv:1308.0062.

\bibitem[CM1]{CM11} D.\ Coman and G.\ Marinescu, 
{\em Equidistribution results for singular metrics on line bundles},
Ann.\ Sci.\ \'Ec.\ Norm.\ Sup\'er.\ (4) \textbf{48} (2015), 497--536. 

\bibitem[CM2]{CM13} D.\ Coman and G.\ Marinescu,
{\em Convergence of Fubini-Study currents for orbifold line bundles}, 
Internat.\ J.\ Math.\ {\bf 24} (2013), 1350051, 27 pp.

\bibitem[CM3]{CM13b} D.\ Coman and G.\ Marinescu,
{\em On the approximation of positive closed currents on compact K\"ahler manifolds}, 
Math.\ Rep.\ (Bucur.) {\bf 15(65)} (2013), no.\ 4, 373--386.

\bibitem[CMM]{CMM}
D.~Coman, X.~Ma, and G.~Marinescu.
\emph{Equidistribution for sequences of line bundles on normal K\"ahler spaces}, 
Geom.\ Topol.\ \textbf{21} (2017), no.\ 2, 923--962.

\bibitem[CMN1]{CMN15} D.\ Coman, G.\ Marinescu, and V.-A.\ Nguy\^en,
{\em H\"older singular metrics on big line bundles and equidistribution},  
Int.\ Math.\ Res.\ Notices {\bf 2016}, no.\ 16, 5048--5075.

\bibitem[CMN2]{CMN17} D.\ Coman, G.\ Marinescu, and V.-A.\ Nguy\^en,
{\em Approximation and equidistribution results for pseudo-effective line bundles}, 
J. Math. Pures Appl. (9) {\bf 115} (2018), 218--236.

\bibitem[DLM]{DLM06} X.~Dai, K.~Liu, and X.~Ma, 
\emph{On the asymptotic expansion of {B}ergman kernel}, 
J.\ Differential Geom.\ \textbf{72} (2006), 1--41.

\bibitem[De]{Delin}
H.~Delin,
{\em Pointwise estimates for the weighted {B}ergman projection kernel in $\C^n$, 
using a weighted $L^2$ estimate for the $\overline\partial$ equation}, 
Ann. Inst. Fourier (Grenoble) {\bf 48} (1998), 967--997.

\bibitem[D1]{D82} J.-P.\ Demailly, 
{\em Estimations $L^2$ pour l'op\'erateur $\overline\partial$ 
d'un fibr\'e holomorphe semipositif au--dessus d'une vari\'et\'e 
k\"ahl\'erienne compl\`ete}, 
Ann.\ Sci.\ \'Ecole Norm.\ Sup.\ {\bf 15} (1982), 457--511.

\bibitem[D2]{D85} J.-P.\ Demailly, 
{\em Mesures de Monge-Amp\`ere et caract\'erisation 
g\'eom\'etrique des vari\'et\'es alg\'ebriques affines}, 
M\'em.\ Soc.\ Math.\ France (N.S.) No.\ 19 (1985), 1--125.

\bibitem[D3]{D90} J.-P.\ Demailly,
{\em Singular Hermitian metrics on positive line bundles}, in 
{\em  Complex algebraic varieties (Bayreuth, 1990)}, 
Lecture Notes in Math.\ 1507, Springer, Berlin, 1992, 87--104.

\bibitem[DMM]{DMM}  T.-C.\ Dinh, X. Ma and G. Marinescu, 
{\em Equidistribution and  convergence  
speed for zeros of holomorphic sections of singular Hermitian line bundles}, 
J.\ Funct.\ Anal.\ \textbf{271} (2016), no.\ 11, 3082--3110.

\bibitem[DMS]{DMS} T.-C.\ Dinh, G.\ Marinescu and V.\ Schmidt, 
{\em Asymptotic distribution of zeros of holomorphic sections
in the non compact setting}, J.\ Stat.\ Phys.\ {\bf 148} (2012), 
113--136.

\bibitem[DNS]{DNS}
T.C. Dinh, V.A. Nguy{\^e}n, and N.~Sibony,
{\em Exponential estimates for plurisubharmonic functions}, 
J.\ Differential Geom.\ {\bf 84} (2010), 465--488.

\bibitem[DS]{DS06} T.-C. Dinh and N. Sibony,
{\em Distribution des valeurs de transformations m\'eromorphes et applications}, 
Comment.\ Math.\ Helv.\ {\bf 81} (2006), 221--258.

\bibitem[DF]{DoFe83}
H. Donnelly and Ch. Fefferman, 
\emph{$L_2$-cohomology and index theorem for the Bergman metric}, 
Ann. of Math. (2) \text{118} (1983), no. 3, 593--618.

\bibitem[EGZ]{EGZ09} P. Eyssidieux, V. Guedj, and A. Zeriahi, 
{\em Singular K\"ahler-Einstein metrics},  
J.\ Amer.\ Math.\ Soc.\ {\bf 22} (2009), 607--639.

\bibitem[FZ]{FZ11} R. Feng and S. Zelditch, 
{\em Large deviations for zeros of $P(\varphi)_2$ random polynomials}, 
J. Stat. Phys. {\bf 143} (2011), no. 4, 619--635.

\bibitem[FN]{FN80} J. E. Forn\ae ss and R. Narasimhan, 
{\em The Levi problem on complex spaces with singularities}, 
Math.\ Ann.\ {\bf 248} (1980), 47--72.

\bibitem[FS]{FS95} J.\ E.\ Forn\ae ss and N. Sibony, 
{\em Oka's inequality for currents and  applications}, 
Math.\ Ann.\ {\bf 301} (1995), 399--419.

\bibitem[GW]{GW14} D.\ Gayet and J.-Y.\ Welschinger, 
\emph{What is the total Betti number of a random
real hypersurface?}, 
J.\ Reine Angew.\ Math.\ \textbf{689} (2014), 137--168.

\bibitem[G]{Gra:62} H.\ Grauert, \emph{{\"Uber Modifikationen
und exzeptionelle analytische Mengen}}, 
Math.\ Ann.\ \textbf{146} (1962), 331--368.

\bibitem[GR1]{GR56} H.\ Grauert and R.\ Remmert, 
{\em Plurisubharmonische Funktionen in komplexen R\"aumen}, 
Math.\ Z.\ {\bf 65} (1956), 175--194.

\bibitem[GR2]{GR84} H.\ Grauert and R.\ Remmert, 
\emph{Coherent Analytic Sheaves}, 
Springer, Berlin, 1984. Grundlehren der Mathematischen 
Wissenschaften, 265, Springer-Verlag, Berlin, 249 pp., 1984. 

\bibitem[GZ]{GZ05} V.\ Guedj and A.\ Zeriahi, 
{\em Intrinsic capacities on compact K\"ahler manifolds}, 
J.\ Geom.\ Anal.\ {\bf 15} (2005), 607--639.

\bibitem[L]{Lin01} N.\ Lindholm, 
{\em Sampling in weighted $L^p$ spaces of entire functions in ${\C}^n$ and 
estimates of the Bergman kernel}, 
J. Funct. Anal. {\bf 182} (2001), 390--426.

\bibitem[LZ]{LuZ13} Z. Lu and S. Zelditch,
\emph{Szeg\"o kernels and Poincar\'e series},  J. Anal. Math. \textbf{130} (2016), 167--184.

\bibitem[MM1]{MM07} X.\ Ma and G.\ Marinescu, 
{\em Holomorphic Morse Inequalities and Bergman Kernels}, 
Progress in Math., vol.\ 254, Birkh\"auser, Basel, 2007, xiii, 422 pp.

\bibitem[MM2]{MM15}
X.\ Ma and G.\ Marinescu, 
{\em Exponential estimate for the asymptotics of Bergman kernels},
Math.\ Ann.\ \textbf{362} (2015), no.\ 3-4, 1327--1347.

\bibitem[N]{Nar62} R.~Narasimhan, 
\emph{The Levi problem for complex spaces II}, 
Math.\ Ann.\ \textbf{146} (1962), 195--216.

\bibitem[NS]{NS14} L.\ Nicolaescu and N.\ Savale, 
\emph{The Gauss-Bonnet-Chern theorem: a probabilistic perspective}, 
Trans. Amer. Math. Soc. {\bf 369} (2017), 2951--2986.

\bibitem[NV]{NoVo:98} S.~Nonnenmacher and A.~Voros, 
\emph{Chaotic eigenfunctions in phase space}, 
J.\ Stat.\ Phys.\ \textbf{92} (1998), no.~3-4, 451--518.

\bibitem[O]{O87} T.\ Ohsawa, 
{\em Hodge spectral sequence and symmetry on compact K\"ahler spaces},  
Publ.\ Res.\ Inst.\ Math.\ Sci.\ {\bf 23} (1987), 613--625.

\bibitem[S]{Sh08} B.\ Shiffman, 
{\em Convergence of random zeros on complex manifolds}, 
Sci.\ China Ser.\ A {\bf 51} (2008), 707--720.

\bibitem[SZ1]{ShZ99} B.\ Shiffman and S.\ Zelditch, 
{\em Distribution of zeros of random and quantum chaotic sections of positive line bundles}, 
Comm.\ Math.\ Phys.\ {\bf 200} (1999), 661--683.

\bibitem[SZ2]{ShZ08} B.\ Shiffman and S.\ Zelditch, 
{\em  Number variance of random zeros on complex manifolds},
Geom.\ Funct.\ Anal.\ {\bf 18} (2008), 1422--1475.

\bibitem[ST]{STr}
M.~Sodin and B.~Tsirelson,
{\em Random complex zeroes. {I}. {A}symptotic normality}, 
Israel J. Math. {\bf 144} (2004), 125--149.

\bibitem[T]{Ti90} G.\ Tian, 
{\em On a set of polarized K\"ahler metrics on algebraic manifolds}, 
J.\ Differential Geom.\ {\bf 32} (1990), 99--130.

\end{thebibliography}
\end{document}